\documentclass{amsart}

\usepackage{fullpage}
\usepackage{color}
\usepackage{graphicx}

 \usepackage{amssymb,latexsym,amsmath,amsthm}

\usepackage[numbers,sort&compress]{natbib}

\usepackage{dsfont}

 \renewcommand{\div}{\mathop{\mathrm{div}}\nolimits}

\newtheorem*{thm*}{Theorem A}

\newtheorem{thm}{Theorem}[section]

\newtheorem{dfn}{Definition}[section]

\newtheorem{lemma}{Lemma}[section]

\newtheorem{prop}{Proposition}[section]

\newtheorem{cor}{Corollary}[section]

\newtheorem*{conj*}{Conjecture A}

\newtheorem*{conj**}{Conjecture B}

 \numberwithin{equation}{section}

\begin{document}

\def\RR{{\mathbb{R}}}

\title{On nonlocal systems with jump processes of finite range and with decays}

\author{Mostafa Fazly}
\author{Changfeng Gui}

\address{Department of Mathematics, The University of Texas at San Antonio, San Antonio, TX 78249, USA}
\email{mostafa.fazly@utsa.edu}
\email{changfeng.gui@utsa.edu}


\maketitle

\begin{abstract}  
We study the following  system of  equations 
$$   L_i(u_i)    =   H_i(u_1,\cdots,u_m)  \quad  \text{in} \ \  \RR^n , $$ 
when $m\ge 1$, $u_i: \RR^n\to \RR$ and $H=(H_i)_{i=1}^m$ is a sequence of general nonlinearities.   The nonlocal operator $L_i$ is given by  
$$L_i(f (x)):=  \lim_{\epsilon\to 0} \int_{\RR^n\setminus B_\epsilon(x) }  [f(x) - f(z)] J_i(z-x) dz,$$
for a sequence of even, nonnegative and measurable jump kernels  $J_i$. We  prove a  Poincar\'{e} inequality for stable solutions of the above system for a general jump kernel $J_i$.  In particular, for the case of scalar equations, that is when $m=1$,  it reads 
\begin{equation*}\label{}
\iint_{  \mathbb R^{2n}}     \mathcal A_y(\nabla_x u)  [\eta^2(x)+\eta^2(x+y)] J(y) dx dy \le  
\iint_{  \mathbb R^{2n}} \mathcal B_y(\nabla_x u) [ \eta(x) - \eta(x+y) ] ^2 J(y) d x dy  , 
  \end{equation*} 
for any $\eta \in C_c^1(\mathbb R^{n})$ and for some  nonnegative 
$ \mathcal A_y(\nabla_x u)$ and $ \mathcal B_y(\nabla_x u)$. This is a counterpart of the celebrated inequality derived   by Sternberg and Zumbrun in \cite{sz} for semilinear elliptic equations  that is used extensively in the literature to establish De Giorgi type results, to study phase transitions and to prove regularity properties.  We then apply this inequality to finite range jump processes and to  jump processes with decays to  prove De Giorgi type results in two dimensions.   In addition, we show that whenever  $H_i(u)\ge 0$ or $\sum_{i=1}^m u_i H_i(u)\le 0$ then  Liouville theorems hold for each $u_i$ in one and two dimensions. Lastly, we provide certain  energy estimates under various assumptions on the jump kernel $J_i$ and a Liouville theorem for the quotient of partial derivatives of $u$.

\end{abstract}


\section{Introduction and main results}\label{secint}

The concepts of diffusion processes and elliptic differential operators are at the centre of the probability theory and partial differential equations theory, respectively. For an elliptic differential operator $L$ in $\mathbb R^n$, there is an associated  diffusion process $X$ such that $L$ is the infinitesimal generator of $X$ and vice versa. As an example, the infinitesimal generator of a isotropically symmetric $\alpha$-stable process in $\mathbb R^n$ with $0<\alpha<2$ is the fractional Laplacian operator $L=(-\Delta)^{\alpha/2}$ that is 
 \begin{equation}\label{Luj}
L u(x) = \lim_{\epsilon\to 0} \int_{\RR^n\setminus B_\epsilon(x) } [u(x) - u(z)] J(x,z) dz,
  \end{equation}   
for  the symmetric jumping kernel $J(x,z) =  \frac{c_{n,\alpha}}{|x-z|^{n+\alpha}}$, where $c_{n,\alpha}$ is a positive constant.  The above nonlocal operator with a measurable symmetric kernel 
 \begin{equation}\label{Jump}
J(x,z)= \frac{c(x,z)}{|x-z|^{n+\alpha}}, 
 \end{equation} 
 when a measurable symmetric function $c(x,z)$ is bounded between two positive constants,  $0<\lambda \le \Lambda$,   is studied expensively in the literature from both theory of partial differential equations and theory of probability points of view, see \cite{ber, bl, bass} and references therein.  Note that there exists a real-valued symmetric $\alpha$-stable jump process $X$ with the jumping kernel $J(x,z)$.  Occasionally, in this article we assume that $J(x,z)$ is restricted in the sense that jumps with size larger than certain number is removed that is 
 \begin{equation}\label{Jumpk}
J(x,z)= \frac{c(x,z)}{|x-z|^{n+\alpha}} \mathds{1}_{\{|x-z|\le \kappa\}}, 
 \end{equation} 
for some positive constant $\kappa$. For a constant $c(x,z)$, the associated $X$ is a finite range, or truncated,  isotropically symmetric $\alpha$-stable process with jumps less than $\kappa$.  From the probability theory point of view,  processes with finite range  jump kernels  are of interests in the literature. We refer interested readers to  \cite{bbck,ckk} and references therein for sharp heat kernel estimates, parabolic Harnack principles and weighted Poincar\'{e} inequalities  of fractional order.   Barlow, Bass and Gui in \cite{bbg} considered jump processes which are not as  restricted as the ones associated to (\ref{Jumpk}). Let $\theta,C>0$ and set 
 \begin{equation}\label{Jumpke}
J(x,z) = \frac{c(x,z)}{|x-z|^{n+\alpha}} \ \ \ \text{when} \ \ \ |x-z|\le \kappa  ,
 \end{equation} 
 and 
  \begin{equation}\label{Jumpker}
\int_{|x-z|>r} J(x,z) dz \le C e^{-\theta r}\ \ \ \text{when} \ \ \ r > \kappa. 
 \end{equation} 
These imply that jumps with size larger than certain number must decay exponentially. However,  jumps with  smaller size  are comparable to the ones in the unrestricted case (\ref{Jump}).     In this article, we study the following multi-component system of nonlocal equations 
 \begin{equation} \label{main}
   L_i(u_i)    =   H_i(u)  \quad  \text{in} \ \  \RR^n , 
  \end{equation}   
  when  $u=(u_i)_{i=1}^m$ for  $u_i\in C^1(\mathbb R^n)$ and  $|\nabla u_i|\in  L^\infty{\mathbb (R^n)}$ and $H=(H_i)_{i=1}^m$ is a sequence of locally Lipschitz functions.   The operator $L_i$ is an integral operator of convolution type 
  \begin{equation}\label{Lui}
L_i(u_i (x))= \lim_{\epsilon\to 0} \int_{\RR^n\setminus B_\epsilon(x) } [u_i(x) - u_i(z)] J_{i}(z-x) dz,  
\end{equation}
where each $J_i$ is a nonnegative measurable even jump kernel.  Inspired by (\ref{Jump}) and (\ref{Jumpk}) we consider these classes of jump kernels in this article, 
 \begin{equation}\label{Jumpi}
 J_i(x-z) =  \frac{c_i(x-z)}{|x-z|^{n+\alpha_i}}, 
 \end{equation} 
 and the finite range processes with jumps 
 \begin{equation}\label{Jumpki}
\frac{c_i(x-z)}{|x-z|^{n+\alpha_i}} \mathds{1}_{\{|x-z|\le \delta_i\}} \le J_i(x-z) \le  \frac{c_i(x-z)}{|x-z|^{n+\alpha_i}} \mathds{1}_{\{|x-z|\le \kappa_i\}}, 
 \end{equation} 
 when a measurable even function $c_i(x-z)$ is bounded between two positive constants $0<\lambda_i\le \Lambda_i$ and $0<\delta_i\le \kappa_i$ for all $1\le i\le m$. We also consider the following jump kernel with decays that is inspired by (\ref{Jumpke})-(\ref{Jumpker}), 
\begin{equation}\label{Jumpkei}
J_i(x,z) = \frac{c_i(x-z)}{|x-z|^{n+\alpha_i}} \ \ \ \text{when} \ \ \ |x-z|\le \kappa_i, 
 \end{equation} 
 and 
  \begin{equation}\label{Jumpkeri}
\int_{r<|x-z|<2r} |J_i(x,z) | dz \le C h_i(r)\ \ \ \text{when} \ \ \ r > \kappa_i,  
 \end{equation} 
where each $0\le h_i\in C(\mathbb R^+)$ with $\lim_{r\to\infty} h_i(r)=0$. We may assume certain decays rates, weaker than (\ref{Jumpker}),  on the sequence function $h=(h_i)_{i=1}^m$ in some of our proofs. 

\subsection{Scalar equations: $m=1$}
Ennio De Giorgi (1978) in \cite{De} proposed the following conjecture that has been a great inspiration for many authors in the field of partial differential equations and differential geometry.   
\begin{conj*} 
Suppose that $u$ is an entire solution of the Allen-Cahn equation that is 
\begin{equation}\label{allen}
\Delta u+u-u^3 =0	 \quad \text{in}\ \  \mathbb{R}^n, 
\end{equation}
satisfying $|u({x})| \le 1$, $\frac{\partial u}{\partial x_n} ({x}) > 0$ for ${x} = ({x}',x_n) \in \mathbb{R}^n$.	
Then, at least in dimensions $n\le 8$ the level sets of $u$ must be hyperplanes, i.e. there exists $g \in C^2(\mathbb{R})$ such that $u({x}) = g(a{x}' - x_n)$, for some fixed $a \in \mathbb{R}^{n-1}$.
\end{conj*}
Ghoussoub and Gui in \cite{gg1} provided the first comprehensive affirmative answer on the De Giorgi conjecture in two dimensions.   In fact their proof is valid for a general Lipschitz nonlinearity $H$  and not necessarily double-well potentials of the following form 
 \begin{equation}\label{scalmainH}
 -\Delta u=H(u) \ \ \text{in} \ \ \ \RR^n. 
 \end{equation} 
  Their proof uses the following, by now  standard, linear Liouville type theorem for elliptic equations in the divergence form, which (only) holds in dimensions one  and two, see \cite{bar, bbg, gg1}. If $\phi>0$, then any solution $\sigma$ of 
\begin{equation}\label{divphisigma}
\div(\phi^2 \nabla \sigma)=0, 
\end{equation}
such that $\phi\sigma$ is bounded, is necessarily constant. This Liouville theorem is then applied to 
 the ratio $\sigma :=   \frac{\partial u}{\partial x_1}  / \frac{\partial u}{\partial x_2}  $ to finish the proof in two dimensions.  Ambrosio and Cabr\'{e} in  \cite{ac}, and later with Alberti in \cite{aac} for a general Lipschitz nonlinearity, provided a proof  for the conjecture  in three  dimensions by noting  that for the linear Liouville theorem to hold,  it suffices that 
\begin{equation}
\int_{B_{2R}\setminus B_R}\phi^2\sigma^2 \leq CR^2, 
\end{equation} 
for any $R>1$.   Then by showing that any solution $u$ such that $\partial_{x_n} u>0$ satisfies the energy estimate 
\begin{equation}
\int_{B_R}\phi^2\sigma^2 \leq CR^{n-1}, 
\end{equation} 
they finished the proof.    Even though the original conjecture remains open in dimensions $4\leq n \leq 8$,  there are various partial, yet groundbreaking, results for dimensions $4\leq n \leq 8$.  In this regard,  Ghoussoub and Gui showed in \cite{gg2} that the conjecture holds in  four and five dimensions for solutions that satisfy certain antisymmetry conditions, and Savin in \cite{sav} established its validity   for $4 \le n \le 8$ under the following additional natural hypothesis on the solution,
 \begin{equation}\label{asymp}
\lim_{x_n\to\pm\infty } u({x}',x_n) = \pm 1 \ \ \text{for} \ \ \ x'\in\mathbb R^{n-1}. 
\end{equation}
In \cite{wang}, Wang provided a new proof for the result of Savin that involves more variational methods in the spirit.     Unlike the above proofs in dimensions $n\leq 5$, the proof of Savin is non-variational and does not use a Liouville type theorem. We refer interested readers to \cite{bar,bbg,bcn,bhm} for more information regrading this Liouville theorem for local semilinear equations.     For the case of $n\ge 9$,  del Pino, Kowalczyk and Wei in \cite{dkw} provided an example that shows the eight dimensions is the critical dimension for the conjecture.   The classical De Giorgi's conjecture with the stability assumption, instead of monotonicity, is known as the {\it stability conjecture} or {\it generalized De Giorgi's conjecture}. 
\begin{conj**}\label{stconj}
Suppose that $u$ is a bounded and stable solution of the Allen-Cahn equation (\ref{allen}). Then,  the level sets are all hyperplanes, at least for dimensions $n \le 7$. 
\end{conj**}
The above conjecture  is known to be true for $n=2$ and not to be true in dimension $n=8$. For two dimensions the same proof given for monotone solutions in \cite{gg1} can be applied for the stability conjecture.  For eight dimensions,  we refer interested readers  to \cite{pw} by Pacard and Wei where they disproved the above conjecture. Note that the stability conjecture remains open for  dimensions $3\le n\le 7$.    In \cite{fsv} and references therein, Farina et al. applied the following geometric Poincar\'{e} type inequality,  given by Sternberg and Zumbrun in \cite{sz,sz1} for stable solutions of (\ref{scalmainH}),   to establish the De Giorgi's conjecture,  among other results,  in two dimensions 
\begin{equation}\label{farinapoin}
   \int_{  \{|\nabla u|\neq 0\}\cap \mathbb R^n}\left(   |\nabla u|^2 \mathcal{A}^2 + | \nabla_T |\nabla u| |^2  \right)\eta^2 \le \int_{\mathbb R^n}|\nabla u|^2   |\nabla \eta|^2, 
  \end{equation} 
for any $\eta \in C_c^1(\mathbb R^n)$ where $\nabla_T$ stands for the tangential gradient along a given level set of $u$ and  $\mathcal{A}^2$ for the sum of the squares of the principal curvatures of such a level set. The above inequality has been frequently used in the literature to establish symmetry results \cite{sv,fsv,mf}, to provide regularity of external solutions \cite{Cabre}, to analyze phase transitions \cite{sz1}, etc.  Consider  the case of scalar equations, when $m=1$ in (\ref{main}), that is  
 \begin{equation} \label{main1}
   L(u) =\lim_{\epsilon\to 0} \int_{\RR^n\setminus B_\epsilon(x) } [u(x) - u(z)] J(z-x) dz   =   H(u)  \quad  \text{in} \ \  \RR^n , 
  \end{equation}  
with an even, nonnegative and measurable kernel $J$. In this article, we prove a counterpart of the above inequality for stable solutions  
\begin{equation}\label{poinsysm1}
\iint_{  \mathbb R^{2n}}     \mathcal A_y(\nabla_x u)  [\eta^2(x)+\eta^2(x+y)] J(y) dx dy \le  
\iint_{  \mathbb R^{2n}} \mathcal B_y(\nabla_x u) [ \eta(x) - \eta(x+y) ] ^2 J(y) d x dy  , 
  \end{equation} 
for any $\eta \in C_c^1(\mathbb R^{n})$ where 
\begin{eqnarray}\label{mathcalA}
\mathcal A_y(\nabla_x u)  &:=&  |\nabla_x u(x)|  |\nabla_x u(x+y)| -\nabla_x u(x)  \cdot \nabla_x u(x+y) , \\
\label{mathcalB}
 \mathcal B_y(\nabla_x u) &:= & |\nabla_x u(x)| | \nabla_x u(x+y)| .
 \end{eqnarray}
 Note that $\mathcal A_y(\nabla_x u)$  and $ \mathcal B_y(\nabla_x u) $ are nonnegative for all $x,y\in\mathbb R^n$. Note also that the nonlinearity $H$ does not appear in both (\ref{farinapoin}) and (\ref{poinsysm1}).   For the particular  kernel $ J(y) =  \frac{c_{n,\alpha}}{  |y|^{n+\alpha}} $, it is by now standard that the fractional Laplacian  operator,  can be denoted  as the Dirichlet-to-Neumann map for an extension function satisfying a higher order elliptic equation in the upper half space with one extra spatial dimension, see \cite{cas} by Caffarelli and Silvestre.  In the light of this extension function, a counterpart of (\ref{poinsysm1}) is given by Sire and  Valdinoci in \cite{sv}.  We also refer interested readers to Cinti and Ferrari in \cite{cf} for a counterpart of (\ref{poinsysm1}).   We then apply this inequality to prove one-dimensional symmetry results for stable solutions of (\ref{main}) in two dimensions. Our approach is different from the ones given very recently by    Hamel,  Ros-Oton,  Sire and Valdinoci in \cite{hrsv}. Authors in \cite{hrsv} studied the scalar equation (\ref{main1}) and, among other results, they proved one-dimensional symmetry of monotone and stable solutions.   Their proof relies on a Harnack inequality and a Liouville theorem for the quotient function $\sigma$. We would like to point out that under an additional assumption of the form of (\ref{asymp}), Savin in \cite{sav2} established the validity of the (weak form of) the
De Giorgi conjecture in a fractional framework. 
 
 \subsection{System of equations: $m\ge1$}

 Ghoussoub and the author in \cite{fg,mf}  considered  the case of system of semilinear equations, 
  \begin{equation}\label{gHm}
 -\Delta u_i=\partial_{i} H(u_1,\cdots,u_m) \ \ \text{in} \ \ \ \RR^n ,
 \end{equation} 
when $\partial_{i} H(u_1,\cdots,u_m)  =\partial_{u_i} H(u_1,\cdots,u_m)$.   Note that for a multi-component solution $u=(u_i)_{i=1}^m$  the concept of monotonicity needs to be adjusted accordingly.   We borrow the following definition from \cite{fg}.    

\begin{dfn}\label{Hmon} A solution $u=(u_k)_{k=1}^m$ of (\ref{main}) is said to be  $H$-monotone if the following hold,
\begin{enumerate}
 \item[(i)] For every $1\le i \le m$, each $u_i$ is strictly monotone in the $x_n$-variable (i.e., $\partial_{x_n} u_i\neq 0$).

\item[(ii)]  For all $i< j$,  we have 
  \begin{equation}\label{huiuj}
\hbox{$\partial_j H_i(u) \partial_n u_i(x) \partial_n u_j (x) > 0$  for all $x\in\mathbb {R}^n$.}
\end{equation}
\end{enumerate}
\end{dfn}
Note also that (ii) implies a combinatorial assumption on the sign of partial derivatives of $H_i$ and we call any system that admits such an assumption as {\it orientable} system.   In other words, for orientable systems there exists a sequence of sign functions $\tau=(\tau_i)_{i=1}^m$ where  $\tau_i\in \{-1,1\}$ such that $\partial_j H_i(u) \tau_i \tau_j >0$.   Here is the notion of stability.  
 \begin{dfn} \label{stable}
A solution $u=(u_i)_{i=1}^m$ of (\ref{main}) is called stable when there exists a sequence of  functions   $\phi=(\phi_i)_{i=1}^m$   such that each $\phi_i$ does not change sign. In addition,  $\phi$ satisfies   the following linearized equation
 \begin{equation} \label{L}
L_i (\phi_i)= \sum_{j=1}^m \partial_j H_i(u)  \phi_j     \ \ \ \text{in}\ \ \mathbb R^n, 
  \end{equation} 
  where   $\partial_j H_i(u)  \phi_j \phi_i >0 $ for all $i<j$ when $1\le i,j\le m$. 
\end{dfn} 
It is straightforward to see that any $H$-monotone solution is a stable solution via differentiating (\ref{main}) with respect to $x_n$ and defining $\phi_i=\partial_{x_n} u_i$.    The  next definition is the notion of symmetric systems, introduced in \cite{mf}. The concept of symmetric systems seems to be crucial for providing De Giorgi type results for system (\ref{main}) with a general nonlinearity. 
\begin{dfn}\label{symmetric} We call system (\ref{main}) symmetric if the matrix of  partial derivatives of all components of $H=(H_i)_{i=1}^m$ that is 
 \begin{equation}
 \mathbb{H}:=(\partial_i H_j(u))_{i,j=1}^{m}, 
 \end{equation}
 is symmetric. 
\end{dfn}
Authors in \cite{fg}, provided  the following geometric inequality for stable solutions of (\ref{gHm}) and then applied this to establish  the De Giorgi's conjecture in two dimensions,
\begin{eqnarray}\label{systempoin}
\sum_{i=1}^m   \int_{\Omega}|\nabla u_i|^2   |\nabla \eta_i|^2&\ge&\sum_{i=1}^m  \int_{ \{|\nabla u_i|\neq 0\}\cap\Omega}\left(   |\nabla u_i|^2 \mathcal{A}_i^2 + | \nabla_T |\nabla u_i| |^2  \right)\eta_i^2\\&&+\sum_{i\neq j}^m \int_{\Omega}  \left( \nabla u_i \cdot\nabla  u_j \eta_i^2 - |\nabla u_i| |\nabla u_j| \eta_i \eta_j\right)\partial_jH_{i}(u),
  \end{eqnarray} 
for any $\eta=(\eta_i)_{i=1}^m \in C_c^1(\Omega)$  and $\mathcal{A}_i^2$ stands for the sum of the squares of the principal curvatures of such a level set.   As it is shown in \cite{abg},  in the absence of $H$-monotonicity and stability there are in fact two-dimensional solutions for  Allen-Cahn systems in two dimensions.   Therefore, the concepts of $H$-monotonicity and stability seem to be crucial in the context.    In this article, we show that the following inequality holds for a stable solution $ u=(u_i)_{i=1}^m$ of (\ref{main}) with a general jump kernel $J=(J_i)_{i=1}^m$,
\begin{eqnarray}\label{poincarein1}
&&\frac{1}{2}\sum_{i=1}^m  \iint_{  \mathbb R^{2n} }   \mathcal A_y(\nabla_x u_i) [ \eta_i^2(x)+\eta^2_i(x+y)] J_i(y) dx dy
\\&& \label{poincarein2} +\sum_{i\neq j}^m \int_{  \mathbb R^{n} }   \partial_{ j} H_i( u)   \left[  |\nabla_{ x}  u_i(x)|  |\nabla_{ x}  u_j(x)| \eta_i(x) \eta_j(x)   -   \nabla_{ x}  u_i(x) \cdot   \nabla_{ x}  u_j(x) \eta_i^2(x) \right] dx
\\&\le&  \label{poincarein3} \frac{1}{2}  \sum_{i=1}^m  \iint_{   \mathbb R^{2n}  } \mathcal B_y(\nabla_x u_i)  \left[ \eta_i(x) - \eta_i(x+y) \right] ^2 J_i(y) d x dy,
  \end{eqnarray} 
   for any sequence of test functions $\eta=(\eta_i)_{i=1}^m$ for $\eta_i \in C_c^1(\mathbb R^{n})$ and $m,n\ge 1$. Note that  counterparts of this  Poincar\'{e} inequality for the fractional Laplacian operator, and the extension problem,  are provided  in \cite{csm,cs1,cs2, ccinti,ccinti1,sv,fs,mf2,dp} and references therein.      We then apply this inequality to establish De Giorgi type results for  stable solutions of (\ref{main}) in two dimensions with a general nonlinearity $H$.

Here is how this article is structured.  Section \ref{secpoin},   is devoted to proof of the Poincar\'{e} inequality (\ref{poincarein1})-(\ref{poincarein3}) for stable solutions with a general jump kernel $J$. In Section \ref{secDeG}, we provide De Giorgi type results  in two dimensions for finite rang jump kernels (\ref{Jumpki}) and for jump kernels with decays (\ref{Jumpkei})-(\ref{Jumpkeri}).  In addition, we show that under extra sign assumptions  $H_i(u)\ge 0$ or $\sum_{i=1}^m u_i H_i(u)\le 0$, Liouville theorems  hold for each $u_i$ in one and two dimensions. In Section \ref{secen}, we provide energy estimates for unrestricted jump kernels (\ref{Jumpi}) and restricted jump kernels (\ref{Jumpki}) and (\ref{Jumpkei})-(\ref{Jumpkeri}).  Lastly, we provide a Liouville theorem for the quotient of partial derivatives of $u$.  

\section{A Poincar\'{e} Inequality for stable solutions}\label{secpoin}

We start this section with the following technical lemma.  
\begin{lemma}\label{fgprop}
Assume that an operator $L$ is given by (\ref{Luj}) with  a measurable symmetric kernel  $J(x,z)=J(x-z)$  that is even.    Then, 
\begin{eqnarray}
&&L(f(x)g(x)) = f(x)L(g(x))+g(x)L(f(x)) - \int_{\mathbb R^n} \left[f(x)-f(z)  \right] \left[g(x)-g(z)  \right] J(x-z) dz,\\ 
&&\int_{\mathbb R^n} g(x)L(f(x)) dx = \frac{1}{2} \int_{\mathbb R^n} \int_{\mathbb R^n}    \left[f(x)-f(z)  \right] \left[g(x)-g(z)  \right] J(x-z) dx dz,  
\end{eqnarray}
where $f,g\in C^1(\mathbb R^n)$ and the integrals are finite. 
\end{lemma}
\begin{proof}  The proof is elementary and we omit it here. 
\end{proof}

We now  prove a stability inequality for solutions of (\ref{main}).  This inequality plays an important role in our proofs. 

\begin{prop}\label{stabineq}  
  Let $u=(u_i)_{i=1}^m$ denote a stable solution of symmetric system (\ref{main}) when  $J_i(x,z)=J_i(x-z)$ is a measurable symmetric kernel    that is even.  Then,  
\begin{equation} \label{stability}
\sum_{i,j=1}^{m} \int_{\RR^n}  \partial_j H_i(u) \zeta_i(x) \zeta_j(x) dx \le  \frac{1}{2} \sum_{i=1}^{m} \int_{\RR^n} \int_{\RR^n}  [\zeta_i(x)- \zeta_i(z)]^2 J_i(x-z) dz dx , 
\end{equation} 
for any $\zeta=(\zeta_i)_{i=1}^m$ where $ \zeta_i \in C_c^1(\mathbb R^n)$ for $1\le i\le m$. 
\end{prop}  
\begin{proof}
Suppose that $u=(u_i)_{i=1}^m$ denote a stable solution of (\ref{main}). Then there exists a sequence of functions $\phi=(\phi_k)_{k=1}^m$  such that   $\partial_j H_i(u)  \phi_j \phi_i >0 $ and 
\begin{equation} \label{Lii}
L_i(\phi_i)= \sum_{j=1}^m \partial_j H_i(u)  \phi_j   \ \ \ \text{in}\ \ \mathbb R^n. 
  \end{equation} 
Multiply both sides with $\frac{\zeta^2_i}{\phi_i}$ where $\zeta=(\zeta_i)_{i=1}^m$ is a sequence of test functions. Therefore, 
\begin{equation} \label{L1}
\sum_{i=1}^m L_i (\phi_i) \frac{\zeta^2_i}{\phi_i}=\sum_{i,j=1}^m \partial_j H_i(u) \frac{\phi_j}{\phi_i} \zeta^2_i    \ \ \ \text{in}\ \ \mathbb R^n.
  \end{equation} 
 Note that the right-hand side can be rewritten as  
\begin{eqnarray}
  \sum_{i,j=1}^{m}   \partial_j H_i(u) \phi_j \frac{\zeta_i^2}{\phi_i}& =&   \sum_{i<j}^{m}   \partial_j H_i(u) \phi_j \frac{\zeta_i^2}{\phi_i} + \sum_{i>j}^{n}   \partial_j H_i(u) \phi_j \frac{\zeta_i^2}{\phi_i} + \sum_{i=1}^m \partial_i H_i(u) {\zeta_i^2}
  \\&=&  \sum_{i<j}^{m}    \left( \partial_j H_i(u) \phi_j \frac{\zeta_i^2}{\phi_i}+ \partial_i H_j(u) \phi_i\frac{\zeta_j^2}{\phi_j} \right)+  \sum_{i=1}^m \partial_i H_i(u) {\zeta_i^2}
   \\&\ge &   \sum_{i<j}^{m}     \partial_j H_i(u)  \phi_i \phi_j    \left(  \frac{\zeta_i^2}{\phi^2_i}+      \frac{\zeta_j^2}{\phi^2_j} \right) + \sum_{i=1}^m \partial_i H_i(u) {\zeta_i^2}
  \\&\ge &  2 \sum_{i<j}^{m}     \partial_j H_i(u)  \zeta_i \zeta_j+ \sum_{i=1}^m \partial_i H_i(u) {\zeta_i^2}
 =  \sum_{i,j=1}^{m}   \partial_j H_i(u)  \zeta_i\zeta_j.
  \end{eqnarray}
  Here, we have used the notion of symmetric systems and Definition \ref{stable}.  From this and (\ref{L1}) we get 
\begin{equation} \label{LL1}
\sum_{i,j=1}^{m}  \int_{\RR^n} \partial_j H_i(u(x))  \zeta_i(x)\zeta_j(x)dx \le  \sum_{i=1}^m  \int_{\RR^n} L (\phi_i(x)) \frac{\zeta^2_i(x)}{\phi_i(x)} dx . 
  \end{equation} 
Applying Lemma \ref{fgprop} for the right-hand side of the above for each $i$ we have
\begin{equation}\label{Lphi}
\int_{\RR^n} L_i (\phi_i(x)) \frac{\zeta^2_i(x)}{\phi_i(x)} dx=
\frac{1}{2} \int_{\RR^n} \int_{\RR^n} [\phi_i(x) - \phi_i(z)] \left[ \frac{\zeta^2_i(x)}{\phi_i(x)}-  \frac{\zeta^2_i(z)}{\phi_i(z)} \right] J_i(x-z) dx dz. 
\end{equation}
Note that for $a,b,c,d\in\mathbb R$ when $ab<0$ we have 
\begin{equation}
(a+b)\left[  \frac{c^2}{a} +  \frac{d^2}{b}  \right] \le (c-d)^2    .
\end{equation}
Since each $\phi_i$ does not change sign, we have $\phi_i(x)\phi_i(z)>0$. Setting  $a=\phi_i(x)$, $b=-\phi_i(z)$, $c=\zeta_i(x)$ and  $d=\zeta_i(z)$ in the above inequality and from the fact that $ab=-\phi_i(x)\phi_i(z)<0$, we  conclude 
\begin{equation}
[\phi_i(x) - \phi_i(z)] \left[ \frac{\zeta^2_i(x)}{\phi_i(x)}-  \frac{\zeta^2_i(z)}{\phi_i(z)} \right] \le [\zeta_i(x)- \zeta_i(z)]^2    .
\end{equation} 
Therefore, 
\begin{equation}
\int_{\RR^n} L_i (\phi_i(x)) \frac{\zeta^2_i(x)}{\phi_i(x)} dx\le  \frac{1}{2} \int_{\RR^n} \int_{\RR^n}  [\zeta_i(x)- \zeta_i(z)]^2 J_i(z-x) dz dx.\end{equation} 
This together with (\ref{LL1}) complete the proof.

\end{proof}

Applying the above stability inequality we prove a Poincar\'{e} type inequality for stable solutions. Note that in the absence of stability,  Poincar\'{e}  inequalities are given in \cite{ber,bbck,bbg} and references therein for various kernels.  

\begin{thm}\label{lempoin}
 Assume that  $m,n\ge 1$ and $ u=(u_i)_{i=1}^m$ is a stable solution of (\ref{main}).  Then,  the   inequality (\ref{poincarein1})-(\ref{poincarein3}) holds where  $\mathcal A_y(\nabla_x u_i)$ and $ \mathcal B_y(\nabla_x u_i) $ are given by (\ref{mathcalA}) and (\ref{mathcalB}). 
 \end{thm}

\begin{proof} Since $u=(u_i)_{i=1}^m$ is a stable solution of (\ref{main}), Propostion \ref{stabineq} implies that the stability (\ref{stability}) holds.  We now test the stability inequality on $\zeta_i(x)=|\nabla _x u_i(x)| \eta_i(x)$ where $\eta=(\eta_i)_{i=1}^m$ is a sequence of test functions with $\eta_i\in C_c^1(\mathbb R^n)$, to get 
 \begin{eqnarray} \label{stability}
&&\sum_{i,j=1}^{m} \int_{\RR^n}  \partial_j H_i(u) |\nabla_x u_i(x)|  |\nabla _x u_j(x)|\eta_i(x) \eta_j(x) dx 
\\&\le&  \frac{1}{2} \sum_{i=1}^{m} \int_{\RR^n} \int_{\RR^n}  [ |\nabla_x u_i(x)|  \eta_i(x)-  |\nabla_x u_j(x+y)| \eta_i(x+y)]^2 J_i (y) dy dx. 
\end{eqnarray} 
Rearranging terms in both sides of the above inequality we obtain 
 \begin{eqnarray} \label{Hijstability}
&&\sum_{i=1}^{m} \int_{\RR^n}  \partial_i H_i(u) |\nabla_x u_i(x)|^2  \eta^2_i(x) dx 
 +\sum_{i\neq j}^{m} \int_{\RR^n}  \partial_j H_i(u) |\nabla_x u_i(x)|  |\nabla _x u_j(x)|\eta_i(x) \eta_j(x) dx 
\\&\le&  
\frac{1}{2} \sum_{i=1}^{m} \int_{\RR^n} \int_{\RR^n}   |\nabla_x u_i(x)|^2  \eta^2_i(x) J_i(y) dy dx
+ \frac{1}{2} \sum_{i=1}^{m} \int_{\RR^n} \int_{\RR^n}   |\nabla_x u_i(x+y)|^2  \eta^2_i(x+y) J_i(y) dy dx
\\&& - \sum_{i=1}^{m} \int_{\RR^n} \int_{\RR^n}  |\nabla_x u_i(x)| |\nabla_x u_j(x+y)| \eta_i(x) \eta_i(x+y) J_i(y) dy dx. 
\end{eqnarray} 
We now apply the equation (\ref{main}). Note that for any index $1\le k\le n$ we have 
\begin{equation}
L_i(\partial_{x_k} u_i(x))=\partial_{x_k} L_i(u_i(x))= \sum_{j=1}^m \partial_j H_i(u(x)) \partial_{x_k} u_j(x) . 
\end{equation}
Multiplying both sides of the above equation with $\partial_{x_k} u_i(x) \eta_i^2(x)$  and integrating we have
\begin{equation}
\sum_{i,j=1}^m \int_{\mathbb R^n} \partial_j H_i(u) \partial_{x_k} u_j(x) \partial_{x_k} u_i(x) \eta_i^2(x)
dx =  \sum_{i=1}^m \int_{\mathbb R^n} \partial_{x_k} u_i(x) \eta_i^2(x) L_i (\partial_{x_k} u_i(x)) dx . 
\end{equation}
We now apply Lemma \ref{fgprop} for the right-hand side of the above equality 
\begin{eqnarray}
&& \sum_{i=1}^m \int_{\mathbb R^n} \partial_i H_i(u) |\nabla_x u_i(x)|^2   \eta_i^2(x) dx 
+ \sum_{i\neq j}^m \int_{\mathbb R^n} \partial_j H_i(u)  \nabla_x u_j(x) \cdot  \nabla_x u_i(x) \eta_i^2(x)
\\&=& 
 \frac{1}{2}  \sum_{i=1}^m \int_{\mathbb R^n}\int_{\mathbb R^n} 
\left [ \partial_{x_k} u_i(x) \eta_i^2(x) -\partial_{x_k} u_i(x+y) \eta_i^2(x+y)\right ]\left[ \partial_{x_k} u_i(x) -\partial_{x_k} u_i(x+y) \right] J_i(y) dx dy  .
\end{eqnarray}
This implies that 
\begin{eqnarray}\label{HijIden}
&& \sum_{i=1}^m \int_{\mathbb R^n} \partial_i H_i(u) |\nabla_x u_i(x)|^2   \eta_i^2(x) dx 
+ \sum_{i\neq j}^m \int_{\mathbb R^n} \partial_j H_i(u)  \nabla_x u_j(x) \cdot  \nabla_x u_i(x) \eta_i^2(x)
\\&=&   \frac{1}{2}  \sum_{i=1}^m \int_{\mathbb R^n}\int_{\mathbb R^n} |\nabla_x u_i(x)|^2   \eta_i^2(x)  J_i(y) dx dy 
 +  \frac{1}{2}  \sum_{i=1}^m \int_{\mathbb R^n}\int_{\mathbb R^n} |\nabla_x u_i(x+y)|^2   \eta_i^2(x+y)  J_i(y) dx dy
 \\&&   -  \frac{1}{2}  \sum_{i=1}^m \int_{\mathbb R^n}\int_{\mathbb R^n} \nabla_x u_i(x) \cdot \nabla_x u_i(x+y)   \eta_i^2(x)  J_i(y) dx dy 
 \\&&   -  \frac{1}{2}  \sum_{i=1}^m \int_{\mathbb R^n}\int_{\mathbb R^n} \nabla_x u_i(x) \cdot \nabla_x u_i(x+y)   \eta_i^2(x+y)  J_i(y) dx dy   . 
\end{eqnarray}
Combining (\ref{HijIden}) and (\ref{Hijstability}) we obtain 
\begin{eqnarray}\label{HijIden2}
 && \sum_{i=1}^{m} \int_{\RR^n} \int_{\RR^n}  |\nabla_x u_i(x)| |\nabla_x u_j(x+y)| \eta_i(x) \eta_i(x+y) J_i(y) dy dx 
\\&& +\sum_{i\neq j}^{m} \int_{\RR^n}  \partial_j H_i(u) \left[|\nabla_x u_i(x)|  |\nabla _x u_j(x)|\eta_i(x) \eta_j(x) - \nabla_x u_i(x) \cdot \nabla _x u_j(x)  \eta_i^2(x)\right]dx 
\\&\le&  
  \frac{1}{2}  \sum_{i=1}^m   \int_{\RR^n} \int_{\RR^n}\nabla_x u_i(x) \cdot \nabla_x u_i(x+y)  \left[ \eta_i^2(x) +\eta_i^2(x+y) \right]  J_i(y) dx dy  .
\end{eqnarray}
Using the fact that $ \eta_i(x) \eta_i(x+y)  =\frac{1}{2} [\eta_i^2(x) +\eta_i^2(x+y)] -  \frac{1}{2} \left[ \eta_i(x) -\eta_i(x+y) \right]^2 $ and regrouping terms we get the desired result. 
 
\end{proof}

For scalar equations,  that is when $m=1$, the following term 
$$
\sum_{i\neq j}^m \int_{  \mathbb R^{n}  }   \partial_{ j} H_i( u)   \left[  |\nabla_{ x}  u_i(x)|  |\nabla_{ x}  u_j(x)| \eta_i(x) \eta_j(x)   -   \nabla_{ x}  u_i(x) \cdot   \nabla_{ x}  u_j(x) \eta_i^2(x) \right] dx
, $$
disappears in the Poincar\'{e} inequality (\ref{poincarein1})-(\ref{poincarein3}). Therefore,  we have the following direct consequence of the above theorem. 

\begin{cor}\label{poincor} 
Let $n\ge 1$,  $m=1$ and $ u$ be a stable solution of (\ref{main}).  Then, the inequality (\ref{poinsysm1})  holds. 
\end{cor}

\section{De Giorgi type results and a Liouville theorem}\label{secDeG} 

In this section,  we provide a one-dimensional symmetry result for bounded stable solutions of symmetric system (\ref{main}) in two dimensions.  We assume that the gradient of solution is globally bounded in the entire space.  To do so,  we apply the Poincar\'{e} inequality given in (\ref{poincarein1})-(\ref{poincarein3}) for an appropriate test function.   We now define the following set, for $1\le i\le m$, that will be used frequently in this section
\begin{equation}
\Gamma_i^0 := \{(x,y)\in\mathbb R^n\times \mathbb R^n ;  \ |\nabla_x u_i(x)|\neq 0 \ \ \text{and}\ \  |\nabla_x u_i(x+y)|\neq 0 \}.
\end{equation}
If $m=1$, for the sake of simplicity, we use the notation $\Gamma^0:=\Gamma_1^0$. 

\begin{thm} Suppose that $u=(u_i)_{i=1}^m$  is a bounded stable solution of symmetric system (\ref{main}) in two dimensions.  Assume also that the jump kernel $J=(J_i)_{i=1}^m$ satisfies either (\ref{Jumpki}) or  (\ref{Jumpkei})-(\ref{Jumpkeri}) with $h_i(r)< C r^{-\theta_i}$ when $\theta_i>3$.   Then,  each $u_i$ must be a one-dimensional function for $i=1,\cdots,m$.  
  \end{thm}
\begin{proof} We provide the proof in two cases.  
\\
\\
\noindent {\bf Case 1. Scalar equations: $m=1$}. Set the following standard test function 
\begin{equation}\label{proofeta}
\eta (x):=\left\{
                      \begin{array}{ll}
                        \frac{1}{2}, & \hbox{if $|x|\le\sqrt{R}$,} \\
                      \frac{ \log {R}-\log {|x|}}{{\log R}}, & \hbox{if $\sqrt{R}< |x|< R$,} \\
                       0, & \hbox{if $|x|\ge R$.}
                                                                       \end{array}
                    \right.
                      \end{equation} 
We now apply inequality (\ref{poinsysm1}), given in Corollary \ref{poincor},   that is 
\begin{equation}\label{ABK}
 \iint_{ \Gamma^0 \cap  \mathbb R^{2n} }     \mathcal A_y(\nabla_x u)  \left [\eta^2(x)+\eta^2(x+y)\right] J_1(y) dx dy  \le C  \iint_{\Omega_R }  \left[ \eta(x) - \eta(y) \right] ^2 J_1(x-y) d x dy,
  \end{equation} 
where C is independent from $R$, and  
$\Omega_R:=\cup_{i=1}^6 \Omega^i_R$ is given by 
\begin{eqnarray}\label{OmegaR}
&& \Omega^1_R:=B_{\sqrt R}\times (B_ R\setminus B_{\sqrt R}), \  \Omega^2_R:=(B_ R\setminus B_{\sqrt R})\times (B_ R\setminus B_{\sqrt R}),  \  \Omega^3_R:=(B_ R\setminus B_{\sqrt R})\times (\mathbb R^n\setminus B_R),\ 
\\&&\label{OmegaR2} \ \Omega^4_R:=B_{\sqrt R}\times (\mathbb R^n\setminus B_R), \  \Omega^5_R:=B_{\sqrt R}\times B_{\sqrt R}, \ 
\Omega^6_R:=(\mathbb R^n\setminus B_R)\times (\mathbb R^n\setminus B_R) . 
\end{eqnarray}
Note that $|\eta(x)-\eta(y)|=0$ on $\Omega^5_R$ and $\Omega^6_R$. Therefore,  from (\ref{ABK}) we get 
\begin{eqnarray}\label{intABR}
\ \ \ \ \   \iint_{ \{\mathbb R^n \times B_{\sqrt R}\} \cap \Gamma^0}  \mathcal A_y(\nabla_x u) J_1(y)  dx dy 
 &\le& C \sum_{i=1}^4  \iint_{\Omega^i_R \cap |x-y|\le \kappa_1 }  \left[ \eta(x) - \eta(y) \right] ^2 J_1(x-y) d x dy  
\\&&+C \sum_{i=1}^4  \iint_{\Omega^i_R \cap |x-y| > \kappa_1 }  \left[ \eta(x) - \eta(y) \right] ^2 J_1(x-y) d x dy  \\&=:&\label{sumgam}
C \sum_{i=1}^4 A_i(R) + C \sum_{i=1}^4 B_i(R). 
  \end{eqnarray} 
We now compute an upper bound for each $ A_i(R)$ and $B_i(R)$ in four steps. To do so, we apply the following straightforward inequality frequently, 
\begin{equation}\label{logab}
|\log b - \log a|^2 \le \frac{1}{ab} |b-a|^2,
\end{equation}
where $a,b\in\mathbb R^+$. 

\noindent {\bf Step 1}: Suppose that  $(x,y)\in \Omega^1_R $.   Note that whenever $x\in B_{\sqrt{R}-\kappa_1}$ and $y\in B_ R\setminus B_{\sqrt R}$ or $x\in B_{\sqrt{R}}$ and $y\in B_R \setminus B_{\sqrt R +\kappa_1}$, we have $|x-y|>\kappa_1$. On these sets and whenever the kernel $J_1$ is with a finite range, ie. (\ref{Jumpki}) holds, then  $B_1(R)=0$.  To find an upper bound for $A_1(R)$, without loss of generality,  we assume that $x\in B_{\sqrt{R}}\setminus  B_{\sqrt{R}-\kappa_1}$ and $y\in B_ {\sqrt{R}+\kappa_1} \setminus B_{\sqrt R}$.  From the definition of the test function $\eta$ we have $\eta(x)=\frac{1}{2}$ and $\eta(y)=1-\frac{\log |y|}{\log R}$. Applying (\ref{logab}) and the fact that $|x|< \sqrt R \le |y|$ we get 
\begin{eqnarray}
|\eta(x) -\eta(y)|^2 &=& \frac{1}{\log^2 R} |\log |y| - \log \sqrt R|^2
\le \frac{1}{\log^2 R} \frac{1}{|y| \sqrt R}  | |y| -  \sqrt R|^2 
\\ &\le& \frac{1}{R\log^2 R}   | |y| -  |x||^2 \le \frac{1}{R\log^2 R}   | y -  x|^2   .
\end{eqnarray}
Substituting this in $A_1(R)$, we get 
\begin{eqnarray}
A_1(R) &\le&   \frac{C}{R\log^2 R}  \int_{B_{\sqrt{R}}\setminus  B_{\sqrt{R}-\kappa_1}} dx   \int_{B_{\kappa_1}} |z|^2 J_1(z) dz  
\\&\le & \frac{C }{R\log^2 R}  \int_{B_{\sqrt{R}}\setminus  B_{\sqrt{R}-\kappa_1}} dx \int_{B_{\kappa_1}}  |z|^{2-n-\alpha_1} dz \le  \frac{C}{(2-\alpha_1)} \frac{\kappa_1^{3-\alpha_1}}{\sqrt R \log^2 R} . 
\end{eqnarray} 
 Note that here we have used the facts that $0<\alpha_1<2$ and $n=2$. Now,  we assume that the jump kernel $J_1$ satisfies (\ref{Jumpkei})-(\ref{Jumpkeri}) with $h_1(r)< C r^{-\theta_1}$ when $\theta_1>3$. Note that the above upper bound for $A_1(R)$ holds. We now provide an upper bound for $B_1(R)$ as 
 \begin{equation}
B_1(R) \le  \frac{C}{R\log^2 R}  \int_{B_{\sqrt{R}}} dx  \sum_{k=1}^\infty \int_{k\kappa_1<|z|<2k\kappa_1} |z|^2 J_1(z) dz \le \frac{C \kappa_1^{2-\theta_1}}{R\log^2 R}  \int_{B_{\sqrt{R}}} dx  \sum_{k=1}^\infty k^{2-\theta_1} \le  \frac{C \kappa_1^{2-\theta_1 }}{\log^2 R}. 
\end{equation}

\noindent {\bf Step 2}: Suppose that  $(x,y)\in \Omega^2_R$.  Due to the symmetry in the domain, without loss of generality we assume that $|x|\le |y|$. Since $x,y\in B_ R\setminus B_{\sqrt R}$, from the definition of $\eta$ we have the following 
\begin{equation}
|\eta(x) -\eta(y)|^2 = \frac{1}{\log^2 R} |\log |y| - \log |x||^2 \le \frac{1}{\log^2 R} \frac{1}{|x| |y|}  | |y| -  |x||^2 \le \frac{1}{|x|^2\log^2 R}   | y -  x|^2   .
\end{equation}
If (\ref{Jumpki}) holds, then  $B_2(R)=0$.  For  $A_2(R)$,  we have 
\begin{eqnarray}
A_2(R) &\le& 
\frac{C}{\log^2 R}  \int_{B_{ R}\setminus B_{\sqrt R}} \frac{1}{|x|^2} dx   \int_{B_{\kappa_1}} |z|^2 J_1(z) dz  
\\&\le&  \frac{C }{\log^2 R}  \int_{\sqrt R}^R r^{n-3}dr \int_{B_{\kappa_1}}  |z|^{2-n-\alpha_1} dz 
\le   \frac{C}{2-\alpha_1} \frac{\kappa_1^{2-\alpha_1}}{\log R}  . 
\end{eqnarray} 
 Note that here we have also used the facts that $0<\alpha_1<2$ and $n=2$. Now suppose that (\ref{Jumpkei})-(\ref{Jumpkeri}) hold with $h_1(r)< C r^{-\theta_1}$ when $\theta_1>3$. Then, the above estimate for $A_2(R)$ holds and we have the following for $B_2(R)$, 
 \begin{eqnarray}\label{B2R}
B_2(R) &\le&  \frac{C}{\log^2 R} \int_{B_{ R}\setminus B_{\sqrt R}} \frac{1}{|x|^2} dx   \sum_{k=1}^\infty \int_{k\kappa_1<|z|<2k\kappa_1} |z|^2 J_1(z) dz
\\&\le& \frac{C \kappa_1^{2-\theta_1}}{\log^2 R}  \int_{\sqrt R}^R r^{n-3}dr   \sum_{k=1}^\infty k^{2-\theta_1} \le  \frac{C \kappa_1^{2-\theta_1 }}{\log R}. 
\end{eqnarray}

\noindent {\bf Step 3}: Suppose that  $(x,y)\in \Omega^3_R $. Note that whenever $x\in B_{R-\kappa_1}\setminus B_{\sqrt R}$ and $y\in \mathbb R^n\setminus B_{R}$ or $x\in B_{{R}}\setminus B_{\sqrt R}$ and $y\in \mathbb R^n \setminus B_{R +\kappa_1}$, we have $|x-y|>\kappa_1$. For these values of $(x,y)$, we have $B_3(R)=0$ provided  (\ref{Jumpki}) holds.  Therefore, without loss of generality we assume that $x\in B_{R}\setminus  B_{R-\kappa_1}$ and $y\in B_ {{R}+\kappa_1} \setminus B_{R}$ for large enough $R$.  From the definition of the test function $\eta$ we have $\eta(x)=1-\frac{\log |x|}{\log R}$ and $\eta(y)=0$. Applying (\ref{logab}) and the fact that $|x|<  R \le |y|$ we get 
\begin{eqnarray}
|\eta(x) -\eta(y)|^2 &=& \frac{1}{\log^2 R} |\log |x| - \log  R|^2 \le \frac{1}{\log^2 R} \frac{1}{|x| R}  | |x| -   R|^2  \le \frac{1}{|x|^2 \log^2 R}   | |y| -  |x||^2 
\\&\le& \frac{1}{|x|^2\log^2 R}   | y -  x|^2  . 
\end{eqnarray}
Substituting this in $A_3(R)$, we get 
\begin{eqnarray}
A_3(R) &\le&  
\frac{C}{\log^2 R}  \int_{B_{ R}\setminus B_{R-\kappa_1}} \frac{1}{|x|^2} dx   \int_{B_{\kappa_1}} |z|^2 J_1(z) dz  \\&\le & \frac{C }{\log^2 R}  \int_{R-\kappa_1}^R r^{n-3}dr \int_{B_{\kappa_1}}  |z|^{2-n-\alpha_1} dz 
\le  \frac{C}{2-\alpha_1} \frac{\kappa_1^{2-\alpha_1}}{\log^2 R} . 
\end{eqnarray} 
 Note that here we have used the facts that $0<\alpha_1<2$ and $n=2$. We now suppose that $J_1$ satisfies (\ref{Jumpkei})-(\ref{Jumpkeri}) with $h_1(r)< C r^{-\theta_1}$ when $\theta_1>3$. Similar arguments as the ones given in (\ref{B2R}) imply that 
  \begin{eqnarray}\label{B3R}
B_3(R) \le   \frac{C \kappa_1^{2-\theta_1 }}{\log R}.
\end{eqnarray}

\noindent {\bf Step 4}: Suppose that  $(x,y)\in \Omega^4_R$.  Note that for  large $R>4(1+\kappa_1)$,  we have $|x-y|>\kappa_1$. This implies that $A_4(R)=0$ for either (\ref{Jumpki}) or  (\ref{Jumpkei})-(\ref{Jumpkeri}).  Note also that $B_4(R)=0$ provided (\ref{Jumpki}).  We now assume that  (\ref{Jumpkei})-(\ref{Jumpkeri}) holds and we provide an estimate for $B_4(R)$. Note that $\eta(x)=\frac{1}{2}$ and $\eta(y)=0$ and  $|x-y|>R-\sqrt R>\kappa_1$
\begin{equation}\label{B4R}
B_4(R) = \frac{1}{2} \int_{ B_{\sqrt R}}  dx   \sum_{k=1}^\infty \int_{k(R-\sqrt R)<|z|<2k (R-\sqrt R)} J_1(z) dz \le \frac{C R}{(R-\sqrt R)^{\theta_1}}  \sum_{k=1}^\infty k^{-\theta_1} 
 \le  \frac{C}{ R^{\theta_1-1} }. 
\end{equation}

 From Step 1-4 and (\ref{intABR}),  we conclude that  
  \begin{equation}\label{}
 \iint_{ \{\mathbb R^2 \times B_{\sqrt R}\} \cap \Gamma^0 }    \mathcal A_y(\nabla_x u) J_1(y)  dx dy 
 \le  \frac{C}{\log R}\ \ \ \text{for large} \ R , 
\end{equation}
where $C$ is independent from $R$.  Sending $R\to\infty$ and applying the fact that 
 $$\mathcal A_y(\nabla_x u):=  |\nabla_x u(x)|  |\nabla_x u(x+y)| -\nabla_x u(x)  \cdot \nabla_x u(x+y) \ge 0,$$ for all $x,y\in\mathbb R^2$, we conclude $\mathcal A_y(\nabla_x u) J_1(y) = 0$ a.e. for all $(x,y)\in\mathbb R^4\cap \Gamma^0$. Therefore,  $\mathcal A_y(\nabla_x u) = 0$ for all $(x,y)\in \{\mathbb R^2 \times  B_{\delta_1} \}\cap \Gamma^0 $. This implies that 
\begin{equation}
|\nabla_x u(x)|  |\nabla_x u(x+y) |=\nabla_x u(x)  \cdot \nabla_x u(x+y). 
\end{equation}
Hence, 
\begin{equation}
\nabla_x u(x) \cdot \nabla_x^\perp u(x+y)=0 \ \ \text{for all} \ \ (x,y)\in \{\mathbb R^2 \times  B_{\delta_1} \}\cap \Gamma^0,
\end{equation}
where $ \nabla_x^\perp$ stands for the skew gradient. This completes the argument. 
\\
\\
\noindent {\bf Case 2. System of equations: $m \ge 2$}.   Since the system (\ref{main}) is {\it orientable}, there exist nonzero functions $\tau_k\in C^1(\mathbb{R}^n)$, $k=1,\cdots,m$, which do not change sign such that 
 \begin{equation}\label{sign}
  \partial_j H_{i}\tau_i\tau_j > 0 \ \ \ \text{for all} \ \ \ 1\le  i<j\le m.
  \end{equation} 
Test the Poincar\'{e} inequality  (\ref{poincarein1})-(\ref{poincarein3}) on $\eta_i(x):=\tau_i \eta(x)$ where  $\eta$ is given by (\ref{proofeta}) and $\tau_i\in\{-1,1\}$.  Therefore, 
\begin{eqnarray}\label{peta1}
&&\frac{1}{2}\sum_{i=1}^m  \iint_{  \mathbb R^{2n}  }   \mathcal A_y(\nabla_x u_i) [ \eta^2(x)+\eta^2(x+y)] J_i(y) dx dy
\\&& \label{peta2} +\sum_{i\neq j} \int_{  \mathbb R^{n}  }  \left[  |\nabla_{ x}  u_i(x)|  |\nabla_{ x}  u_j(x)|   -   \nabla_{ x}  u_i(x) \cdot   \nabla_{ x}  u_j(x) \tau_i \tau_j  \right]  \partial_{ j} H_i( u)  \tau_i \tau_j  \eta^2(x)  dx
\\&\le&  \label{peta3} C  \sum_{i=1}^m  \iint_{  \Omega_R  }  \left[ \eta(x) - \eta(y) \right] ^2 J_i(x-y) d x dy,
  \end{eqnarray} 
where $\Omega_R$ is given by (\ref{OmegaR})-(\ref{OmegaR2}).  Applying similar arguments as the ones given in Case 1, for each index $i$, one can see that (\ref{peta3}) approaches to zero when $R$ tends to infinity.  Note that the integrand in (\ref{peta1})  due to fact  that each $\mathcal A_y(\nabla_x u_i) $ is nonnegative. Note also that the integrand in (\ref{peta2}) is also nonnegative, since the system is orientable. Hence, both (\ref{peta1}) and (\ref{peta2}) must be zero. This implies that for all $1\le i\le m$, 
  \begin{equation}\label{uinabla}
\nabla_x u_i(x) \cdot \nabla_x^\perp u_i(x+y)=0 \ \ \text{for all} \ \ (x,y)\in \{\mathbb R^2 \times  B_{\delta_1} \}\cap \Gamma^i, 
\end{equation}
and   for all   $1\le i\neq j\le m$, 
  \begin{equation}\label{uiujnabla}
 |\nabla_{ x}  u_i(x)|  |\nabla_{ x}  u_j(x)|   =\nabla_{ x}  u_i(x) \cdot   \nabla_{ x}  u_j(x)   \tau_i \tau_j  \ \ \text{for all} \ \ x\in\mathbb R^2. 
\end{equation}
These imply that each $u_i$ must be a one-dimensional function and the angle between $\nabla_{ x}  u_i(x) $ and $\nabla_{ x}  u_j(x) $ is $\arccos ( \tau_i \tau_j )$. 

\end{proof}

In what follows,  we provide a Liouville theorem in one and two dimensions for bounded solutions of (\ref{main}) under certain sign assumptions on the nonlinearity $H$.  

\begin{thm}\label{liopo}
Let $m\ge 1$ and $n\le 2$. Suppose that $ u=(u_i)_{i=1}^m$ is a bounded solution of (\ref{main}) when the jump kernel $J=(J_i)_{i=1}^m$ satisfies either (\ref{Jumpki}) or  (\ref{Jumpkei})-(\ref{Jumpkeri}) with $h_i(r)< C r^{-\theta_i}$ when $\theta_i>3$. If $H_i(u) \ge0$ for an index $1\le i\le m$, then $u_i$ must be constant.   In addition, if $\sum_{i=1}^m u_i H_i(u)\le 0$, then all $u_i$ must be constant. 
\end{thm}

\begin{proof}
Consider the standard test function $\eta$ when $\eta=1$ in $\overline {B_R}$ and $\eta=0$ in $\overline{\mathbb R^n\setminus B_{2R}}$ with $\eta\in C_c^1(\mathbb R^n)$ and $||\nabla \eta||_{\infty}<C R^{-1}$ in $\overline{B_{2R}\setminus B_R}$ for $R>\kappa^*$. Suppose that  $H_i(u) \ge0$. Multiply the $i^{th}$ equation of (\ref{main}) with $(u_i(x) - ||u_i||_{\infty}) \eta^2(x)$ and integrate to get 
\begin{equation}
\int_{\mathbb R^n} (u_i(x) - ||u_i||_{\infty}) \eta^2(x) L(u_i(x)) dx \le 0.
\end{equation}
From Lemma \ref{fgprop},   we have
\begin{eqnarray}\label{mathcal1}
0 &\ge& \iint_{\mathbb R^{2n}} \mathcal C_{x,y}(u_i) [ (u_i(x)  - || u_i||_{\infty}) \eta^2(x) - (u_i(y)  - || u_i||_{\infty}) \eta^2(y)  ] J_i(x-y) dx dy
\\&=&  \label{mathcal2} \frac{1}{2}\iint_{\mathbb R^{2n}} \mathcal C_{x,y}(u_i)  \left(  \mathcal C_{x,y}(u_i)   \eta^2(x)+    \left[ u_i(y) - ||u_i||_{\infty}  \right ] [\eta^2(x)- \eta^2(y) ]   \right) J_i(x-y) dx dy
\\&& \label{mathcal1} + \frac{1}{2}\iint_{\mathbb R^{2n}}   \mathcal C_{x,y}(u_i)   \left(   \mathcal C_{x,y}(u_i)   \eta^2(y)+    \left[ u_i(x) - ||u_i||_{\infty}  \right ] [\eta^2(x)- \eta^2(y) ]   \right)J_i(x-y) dx dy   , 
\end{eqnarray}
when $ \mathcal C_{x,y}(u_i)  :=u_i(x)-u_i(y)$. Rearranging the terms, we conclude 
\begin{equation}
\iint_{\mathbb R^{2n}}
\mathcal C^2 _{x,y}(u_i)  [\eta^2(x) + \eta^2(y) ] J_i(x-y) dx dy  \le 4 || u_i||_{\infty} \iint_{\Gamma_R } | \mathcal C_{x,y}(u_i)|   |\eta^2(x)-\eta^2(y)| J_i(x-y) dx dy  ,
 \end{equation}
for $\Gamma_R=\cup_{i=1}^6\Gamma^i_R$ when 
\begin{eqnarray}\label{gamma2}
&&  \Gamma^1_R:=B_{ R}\times (B_ {2R}\setminus B_{R}), \  \Gamma^2_R:=(B_ {2R}\setminus B_{ R})\times (B_ {2R}\setminus B_{ R}), \ \Gamma^3_R:=(B_ {2R}\setminus B_{ R})\times (\mathbb R^n\setminus B_{2R}),\ 
\\&& \label{gamma3}  \Gamma^4_R:=B_{ R}\times (\mathbb R^n\setminus B_{2R}), \  \Gamma^5_R:=B_{R}\times B_{R},  \ \Gamma^6_R:=(\mathbb R^n\setminus B_{2R})\times (\mathbb R^n\setminus B_{2R}).  
\end{eqnarray}
 We now apply the Cauchy-Schwarz inequality to get 
\begin{eqnarray}
\ \ \ \ \iint_{\mathbb R^{2n}}  \mathcal C^2_{x,y}(u_i)   [\eta^2(x) + \eta^2(y) ]  J_i(x-y) dx dy &\le&
\label{Gamma2etaK} C \left[\iint_{\Gamma_R}  \mathcal C^2_{x,y}(u_i)   [\eta^2(x)+\eta^2(y)] J_i(x-y) dx dy \right]^{1/2}
\\&& 
\left[ \iint_{\Gamma_R}  |\eta(x)-\eta(y)|^2 J_i(x-y) dx dy \right]^{1/2} , 
 \end{eqnarray}
where we have used the fact that $  |\eta(x)+\eta(y)|^2 \le 2   [\eta^2(x)+\eta^2(y)]$.  Note that for $(x,y)\in\Gamma^5_R\cup\Gamma^6_R$,  we have $|\eta(x)-\eta(y)|=0$.   Therefore, 
\begin{eqnarray}\label{C2Keta}
\ \ \ \ \ \ \ \   \iint_{\mathbb R^{2n}} \mathcal C^2_{x,y}(u_i)    [\eta^2(x) + \eta^2(y) ]  J_i(x-y) dx dy &\le& C 
\sum_{k=1}^4 \iint_{\Gamma^k_R \cap \{|x-y| \le \kappa_i\} }  \left [ \eta(x) - \eta(y) \right]^2 J_i(x-y) dx dy
\\&&+C 
\sum_{k=1}^4 \iint_{\Gamma^k_R \cap \{|x-y| > \kappa_i\} }  \left [ \eta(x) - \eta(y) \right]^2 J_i(x-y) dx dy
\\&=&  C  \sum_{k=1}^4 E_k(R) + C  \sum_{k=1}^4 F_k(R), 
\end{eqnarray}
where $C$ is a positive constant and it is independent from $R$. We are now ready to compute an upper bound for each $E_k$ and $F_k$.  Note that   for    $(x,y)\in\cup_{i=1}^3 \Gamma^i_R$,  we have 
\begin{equation}
|\eta(x)- \eta(y)|^2 \le C R^{-2} |x-y|^2,
\end{equation}
and for $(x,y)\in\Gamma^4_R$, we have $|\eta(x)- \eta(y)|=1$.  We now provide upper bounds for $E_k,F_k$ in various steps. 
 
\noindent {\bf Step 1}: Suppose that  $(x,y)\in \Gamma^1_R $. First we assume that the jump kernel $J_i$ satisfies  (\ref{Jumpki}). If $x\in B_{R-\kappa_i}$ and $y\in B_ {2R}\setminus B_{R}$ or $x\in B_{R}$ and $y\in B_{2R} \setminus B_{R +\kappa_i}$, we have $|x-y|>\kappa_i$.  For such $(x,y)$,  we have $F_1(R)=0$.  We now provide an upper bound for $E_1(R)$. 
\begin{eqnarray}
E_1(R)&\le& C R^{-2} \int_{B_R\setminus B_{R-\kappa_i}} \int_{B_{R+\kappa_i}\setminus B_R}  |x-y|^2 J_i(x-y) dy dx 
\\ &\le&  CR^{-2} \int_{B_R\setminus B_{R-\kappa_i}} dx \int_{B_{\kappa_i}} |z|^{2-n-\alpha_i} dz \le \frac{C \kappa_i^{3-\alpha_i}}{2-\alpha_i}  R^{-1}   . 
\end{eqnarray}
Now,  suppose that (\ref{Jumpkei})-(\ref{Jumpkeri}) hold with $h_i(r)< C r^{-\theta_i}$ when $\theta_i>3$. Therefore, the above estimate holds for $E_1(R)$ and 
 \begin{equation}
F_1(R) \le  CR^{-2}  \int_{B_R} dx  \sum_{k=1}^\infty \int_{k\kappa_i<|z|<2k\kappa_i} |z|^2 J_i(z) dz
\le \frac{C \kappa_i^{2-\theta_i}}{R^2}  \int_{B_{R}} dx  \sum_{k=1}^\infty k^{2-\theta_i} \le  C \kappa_i^{2-\theta_i } , 
\end{equation} 
where we have used $\theta_1>3$,  $n=2$ and $0<\alpha_i<2$.

\noindent {\bf Step 2}: Suppose that  $(x,y)\in \Gamma^2_R$. Note that for the jump kernel $J_i$  satisfying  (\ref{Jumpki}) we have $F_2(R)=0$ and 
\begin{eqnarray}
E_2(R) \le CR^{-2} \int_{B_ {2R}\setminus B_{ R}} dx \int_{B_{\kappa_i}} |z|^{2-n-\alpha_i} dz \le \frac{C \kappa_i^{2-\alpha_i}}{2-\alpha_i}    . 
\end{eqnarray}
For kernels satisfying (\ref{Jumpkei})-(\ref{Jumpkeri}),  the above estimate holds for $E_2(R)$  and 
\begin{equation}\label{F2R}
F_2(R) \le  CR^{-2}    \int_{B_ {2R}\setminus B_{ R}} dx   \sum_{k=1}^\infty \int_{k\kappa_i<|z|<2k\kappa_i} |z|^2 J_i(z) dz  \le  C \kappa_i^{2-\theta_i }  . 
\end{equation}

\noindent {\bf Step 3}: Suppose that  $(x,y)\in \Gamma^3_R $.  Let $x\in B_{2R-\kappa_i}\setminus B_{ R}$ and $y\in \mathbb R^n\setminus B_{2R}$ or $x\in B_{{2R}}\setminus B_{ R}$ and $y\in \mathbb R^n \setminus B_{2R +\kappa_1}$, we have $|x-y|>\kappa_i$. Therefore, $F_3(R)=0$ when  $J_i$ satisfies  (\ref{Jumpki}).     We have the following estimate for $E_3(R)$, 
\begin{eqnarray}
E_3(R) \le C R^{-2} \int_{B_{2R}\setminus B_{2R-\kappa_i}} \int_{B_{2R+\kappa_i}\setminus B_{2R}}  |x-y|^2 J_i(x-y) dy dx 
 \le \frac{C \kappa_i^{3-\alpha_i}}{2-\alpha_i}  R^{-1}   .  
\end{eqnarray}
 In the above,  we have used the facts that $0<\alpha_i<2$ and $n=2$. Note that a similar estimate as (\ref{F2R}) holds for $F_3(R)$.

 \noindent {\bf Step 4}: Suppose that  $(x,y)\in \Gamma^4_R $.  Let $R$ be large enough that is $R>\kappa^*$.  For such $(x,y)$,  we have $|x-y|>R>\kappa_i$ for $1\le i\le m$. This implies that for kernels satisfying (\ref{Jumpki}),   we have $E_4(R)=F_4(R)=0$.    We now assume that  (\ref{Jumpkei})-(\ref{Jumpkeri}) holds and we provide an estimate for $F_4(R)$. Note that $\eta(x)=1$ and $\eta(y)=0$ and  $|x-y|>R>\kappa_i$, 
\begin{equation}\label{F4R}
F_4(R) =  \int_{ B_{ R}}  dx   \sum_{k=1}^\infty \int_{k R<|z|<2kR} J_i(z) dz  \le \frac{C R^2}{R^{\theta_i}}  \sum_{k=1}^\infty k^{-\theta_i} 
 \le  \frac{C}{ R^{\theta_i-2} }. 
\end{equation}

From the above steps and (\ref{C2Keta}) we get 
 \begin{equation}\label{}
\iint_{\mathbb R^{2n}} \mathcal C^2_{x,y}(u_i)   [\eta^2(x) + \eta^2(y) ]  J_i(x-y) dx dy \le C   , 
\end{equation}
where $C $ is a positive constant that is independent from $R$. From this, definition of $\eta$ and (\ref{Gamma2etaK}) we get 
\begin{equation}\label{}
\iint_{\mathbb R^{2n}}  \mathcal C^2_{x,y}(u_i)  J_i(x-y) dx dy =0.  
\end{equation}
This implies that  $ \mathcal C^2_{x,y}(u_i)  J_i(x-y) =0$ a.e. $(x,y)\in\mathbb R^2\times\mathbb R^2$.  Therefore, $\mathcal C^2_{x,y}(u_i) =|u_i(x)-u_i(y)|^2=0$  for $x\in\mathbb R^2$ and $y\in B_{\delta_i(x)}$. This implies that  $u_i$ is constant. Note that the case of $\sum_{i=1}^m u_i H_i(u)\le 0$ is very similar and we omit the proof.

\end{proof}

\section{Energy estimates and a Liouville theorem}\label{secen}

\subsection{Energy Estimates}
For a domain $\Omega\subset\mathbb R^n$, the associated energy functional to (\ref{main}) is
\begin{equation}
\mathcal E_J(u,\Omega):=\mathcal K_J(u,\Omega) - \int_{\Omega} \tilde H(u) dx, 
\end{equation}
when $\partial_i \tilde H(u)=H_i(u)$ and   the term $\mathcal K_J$ is given by 
\begin{equation}
\mathcal K_J(u,\Omega):= \frac{1}{4} \sum_{i=1}^m \iint_{\mathbb R^{2n}\setminus (\mathbb R^n\setminus \Omega)^2 }   |  u_i(x) -u_i(y) |^2 J_i(x-y) dy dx. 
\end{equation}
Since the jump kernel $J$ is even,  $\mathcal K_J$ becomes 
\begin{eqnarray}
\mathcal K_J(u,\Omega) &=& \frac{1}{4} \sum_{i=1}^m \int_{\Omega} \int_{\Omega}   |  u_i(x) -u_i(y) |^2 J_i(x-y) dy dx \\&&+ \frac{1}{2} \sum_{i=1}^m \int_{\Omega} \int_{\mathbb{R}^n\setminus \Omega}  |  u_i(x) -u_i(y) |^2 J_i(x-y) dy dx. 
\end{eqnarray}
Here we provide the  notion of layer solutions. This is a counterpart of (\ref{asymp}). 
\begin{dfn}\label{layer}
We say that $u=(u_i)_{i=1}^m$ is a layer solution of (\ref{main}) if $ u$ is a bounded solution of (\ref{main}) such that for each index $i$ the directional derivative $\partial_{x_n} u_i (x)$ does not change sign and 
 \begin{equation}\label{asympsys}
\lim_{x_n\to \pm\infty}  u(x',x_{n})= \omega^{\pm}
 \ \ \text{for} \ \ \ x'\in\mathbb R^{n-1}, 
\end{equation} when $\omega^{\pm}=(\omega_i^{\pm})_{i=1}^m$ and $\omega_i^{\pm}\in\mathbb R$.  
\end{dfn}
For the case of layer solutions of (\ref{main}) we assume that   $\tilde H(\omega^+)=\tilde H(\omega^-)$. This refers to multi-well potentials that is of great interests in this context.  Here is how we justify this assumption when $n=1$ and $m\ge 1$. Multiply both sides of (\ref{main}) with $u'_i(x)$ and integrate to get 
\begin{equation}
\sum_{i=1}^m \int_{\mathbb R}u'_i(x) L_i(u_i (x)) dx = \sum_{i=1}^m \int_{\mathbb R} H_i(u)  u'_i(x) dx = \int_{\mathbb R}   \left (\tilde H(u(x))  \right)' dx    . 
\end{equation} 
From Lemma \ref{fgprop} we have 
\begin{equation}
\frac{1}{4} \sum_{i=1}^m  \int_{\mathbb R}  \left ( \int_{\mathbb R} \frac{d}{dx} [u_i(x) - u_i(x+y)]^2 dx \right) J_i(y) dy = \tilde H(\omega^+)-\tilde H(\omega^-)    . 
\end{equation} 
 If the left-hand side of the above is finite,  then it vanishes. Therefore,  $\tilde H(\omega^+)=\tilde H(\omega^-)$.  We now provide an energy estimate for kernels (\ref{Jumpki}) and (\ref{Jumpkei})-(\ref{Jumpkeri}). 

\begin{thm}\label{thmlayerK1}
Let $n,m\ge 1$. Suppose that $ u=(u_i)_{i=1}^m$ is a bounded $H$-monotone layer solution of (\ref{main}) with  $\tilde H(\omega^+)=0$.  Assume also that  the jump kernel $J=(J_i)_{i=1}^m$ satisfies either (\ref{Jumpki}) or  (\ref{Jumpkei})-(\ref{Jumpkeri}) with $h_i(r)< C r^{-\theta_i}$ when $\theta_i>2$.  Then,     
\begin{equation}\label{EKRn}
\mathcal E_J(u,B_R) \le C  R^{n-1} \ \ \text{for} \ \  R>R^*, 
\end{equation}
where the positive constant $C$ is independent from $R$ but may depend on $\kappa_i,\alpha_i,\theta_i$. 
\end{thm}

\begin{proof}
Consider the shift function $u_i^t( x):=u_i( x',x_n+t)$ for $( x',x_n)\in\mathbb R^{n}$ and $t\in\mathbb R$. It is straightforward to see that $u^t=(u^t_i)_{i=1}^m$ is a solution of (\ref{main}).  Consider the energy functional for the shift function $u^t$
\begin{eqnarray}
\mathcal E_J(u^t,B_R) &=& \mathcal K_J(u^t,B_R) - \int_{B_R} \tilde H(u^t) dx  
\\&=& \frac{1}{4} \sum_{i=1}^m \int_{B_R} \int_{B_R}   |  u^t_i(x) -u^t_i(y) |^2 J_i(x-y) dy dx \\&&+ \frac{1}{2} \sum_{i=1}^m \int_{B_R} \int_{\mathbb{R}^n\setminus B_R}  |  u^t_i(x) -u^t_i(y) |^2 J_i(x-y) dy dx - \int_{B_R} \tilde H(u^t) dx, 
\end{eqnarray}
where $R>R^*$.  We now differentiate the energy functional  in terms of parameter $t$ to get
\begin{eqnarray}
\partial_t\mathcal E_J(u^t,B_R) &=& \partial_t \mathcal K_J(u^t,B_R) - \sum_{i=1}^m\int_{B_R}  H_i (u^t) \partial_t u_i^t dx  
\\&=& \frac{1}{2} \sum_{i=1}^m \int_{B_R} \int_{B_R}   [  u^t_i(x) -u^t_i(y) ]  [  \partial_t u^t_i(x) -\partial_t u^t_i(y) ]  J_i(x-y) dy dx \\&&+ \sum_{i=1}^m \int_{B_R} \int_{\mathbb{R}^n\setminus B_R}  [  u^t_i(x) -u^t_i(y) ]  [ \partial_t u^t_i(x) 
- \partial_t u^t_i(y) ] J_i(x-y) dy dx \\&&  - \sum_{i=1}^m\int_{B_R}  H_i (u^t) \partial_t u_i^t dx   . 
\end{eqnarray}
Straightforward computations show that 
\begin{eqnarray}
\partial_t\mathcal E_J(u^t,B_R) &=& \sum_{i=1}^m \int_{B_R} \partial_tu_i(x) L_i( u^t(x)) dx - \sum_{i=1}^m\int_{B_R}  H_i (u^t) \partial_t u_i^t dx  
\\&& + \sum_{i=1}^m  \int_{\mathbb{R}^n\setminus B_R} \int_{B_R}  [  u^t_i(x) -u^t_i(y) ] \left[  \partial_t u^t_i(x) \right] J_i(x-y) dy dx  .
\end{eqnarray}
Since $u^t$ solves (\ref{main}), we can simplify the above as 
\begin{equation}
\partial_t \mathcal E_J(u^t,B_R) =  \sum_{i=1}^m  \int_{\mathbb{R}^n\setminus B_R} \int_{B_R}  [  u^t_i(x) -u^t_i(y) ]  \left[ \partial_t u^t_i(x) \right] J_i(x-y) dy dx    .
\end{equation}
Note that  $ \mathcal E_J(u,B_R)= \mathcal E_J(u^T,B_R)- \int_0^T \partial_t \mathcal E_J(u^t,B_R) dt$ for every $T>0$. Consider disjoint sets of indices $\Gamma$ and $\Lambda$ such that  $\Gamma\cup \Lambda=\{1,\cdots,m\}$ and   $ \partial_t u_\gamma ^t >0> \partial_t u_\lambda ^t$ for $\gamma\in \Gamma$ and $\lambda\in \Lambda$. Then the above reads 
\begin{eqnarray}\label{EKT}
\mathcal E_J(u, B_R) &\le& \mathcal E_J(u^T,B_R) +   \sum_{\gamma}  \int_{\mathbb{R}^n\setminus B_R} \int_{B_R}  \int_0^T  |  u^t_\gamma(x) -u^t_\gamma(y) |  \left[ \partial_t u^t_\gamma(x) \right] J_\gamma(x-y) dt dy dx  
\\&&+ \sum_{\lambda} \int_{\mathbb{R}^n\setminus B_R} \int_{B_R}  \int_0^T   |  u^t_\lambda(x) -u^t_\lambda(y) |  \left[  - \partial_t u^t_\lambda(x) \right] J_\lambda(x-y) dt dy dx  . 
\end{eqnarray}
Note that $|u_i^t(x)- u_i^t(y)| \le C |x-y|$ and $ \lim_{T\to\infty}\mathcal E_J(u^T,B_R)=\mathcal E_J(\omega^+,B_R)=0 $. Therefore, from the boundedness of $u$,  we conclude  
\begin{equation}\label{Pik}
\mathcal E_J(u, B_R) \le  C \sum_{k=1}^3 \sum_{i=1}^m  \iint_{\Pi^k_{R,\kappa_i}}  |x-y|  J_i(x-y)  dy dx =:  C \sum_{k=1}^3 \sum_{i=1}^m  G_{k,i}(R),  
\end{equation}
for domain decompositions   
\begin{equation}
\Pi^1_{R,\kappa_i}:=B_{R-\kappa_i}\times (B_{R+\kappa_i}\setminus B_R), \ \Pi^2_{R,\kappa_i}:=B_{R}\times (\mathbb{R}^n\setminus B_{R+\kappa_i}), \  \Pi^3_{R,\kappa_i}:=(B_{R}\setminus B_{R-\kappa_i})\times(B_{R+\kappa_i}\setminus B_{R})  .  
\end{equation}
Suppose that (\ref{Jumpki})  holds. Note that for $(x,y)$ in  both $\Pi^1_{R,\kappa_i}$ and $\Pi^2_{R,\kappa_i}$,  we have $|x-y|>k_i$. Therefore, $G_{1,i}(R)$ and $G_{2,i}(R)$ are identically zero. We now compute an upper bound for $G_{3,i}(R)$. Straightforward computations show that for $n\ge1$
\begin{equation}\label{IntJ2R}
\int_{ B_{R}\setminus B_{R-\kappa_i}   } \int_{  B_{R+\kappa_i}\setminus B_{R}  }   |x-y|^{1-n-\alpha_i}   dy dx
\le  C(\alpha_i,\kappa_i) R^{n-1} , 
\end{equation}
where $C(\alpha_i,\kappa_i)$ is a positive constant for $0<\alpha_i<2$ and it is given by 
\begin{equation}\label{}
 C(\alpha_i,\kappa_i) := \left\{ \begin{array}{lcl}
\hfill \kappa_i   \ \ \text{for}\ \ \ \alpha_i=1,\\   
\hfill  \left[ \frac{\alpha_i+1}{\alpha_i(2-\alpha_i)} \right] \left[ \frac{2^{1-\alpha_i}-1}{1-\alpha_i} \right]   \kappa_i^{2-\alpha_i}      \ \ \text{for}\ \ \alpha_i\neq 1. 
\end{array}\right.
\end{equation}
This implies that 
\begin{equation}\label{G3R}
G_{3,i}(R) \le C(\alpha_i,\kappa_i) R^{n-1} .
\end{equation}
 We now assume that the jump kernel $J=(J_i)_{i=1}^m$ satisfies  (\ref{Jumpkei})-(\ref{Jumpkeri}). Note that the latter assumption on $G_{3,i}(R)$ holds. We now compute upper bounds for $G_{1,i}(R)$ as 
\begin{equation}
G_{1,i}(R) \le  C  \int_{B_{R+\kappa_i}\setminus B_R} dx  \sum_{k=1}^\infty \int_{k\kappa_i<|z|<2k\kappa_i} |z| J_i(z) dz
\le C \left[\kappa_i^{2-\theta_i} \right] \left[ \sum_{k=1}^\infty k^{1-\theta_i} \right]  R^{n-1}   \le   C \left[\kappa_i^{2-\theta_i}\right] R^{n-1} , 
\end{equation} 
where we have used  $h_i(r)< C r^{-\theta_i}$ for  $\theta_i>2$.    For  $G_{2,i}(R)$, we have 
\begin{eqnarray}
G_{2,i}(R) &\le&   \int_{B_{R}}   \int_{|x-y|>R+\kappa_i- |x|}   |y-x| J_i(y-x)  dy dx 
\\&=& \int_{B_{R}}  \sum_{k=1}^\infty \int_{k(R+\kappa_i- |x|)<|x-y|<2k(R+\kappa_i- |x|)}   |y-x| J_i(y-x)  dy dx 
\\&\le&   \int_{B_{R}}  (R+\kappa_i - |x|)^{1-\theta_i} dx  \left[ \sum_{k=1}^\infty k^{1-\theta_i} \right]\le C R^{n-1} \int_0^R  (R+\kappa_i - r)^{1-\theta_i} dr
\\&=&C\left[ \frac{\kappa_i^{2-\theta_i}}{\theta_i-2}  - \frac{(R+\kappa_i)^{2-\theta_i} }{\theta_i - 2} \right] R^{n-1} \le C\left[ \frac{\kappa_i^{2-\theta_i}}{\theta_i-2} \right] R^{n-1}    ,
\end{eqnarray} 
when $C$ is a positive constant that is independent from $R$.  

\end{proof}
Note that the upper bound in the energy estimate (\ref{EKRn}) is $R^{n-1}$ for all parameters $0<\alpha_i<2$.     However, for jump kernels with intensity (\ref{Jumpi}) we have to consider three different cases $1<\alpha_*<2$, $\alpha_*=1$ and $0<\alpha_*<1$. For the case of fractional Laplacian operator and $m=1$ such an energy estimate is given in \cite{ccinti1,ccinti,cser, psv} and references therein.  For the case of systems, we refer interested readers to \cite{mf2,fs}.  

\begin{thm}\label{thmlayerK2}
Let $n,m\ge 1$. Suppose that $ u=(u_i)_{i=1}^m$ is a bounded $H$-monotone layer solution of (\ref{main}) with  $\tilde H(\omega^+)=0$.  Assume also that the kernel $J_i$ satisfies  (\ref{Jumpi})  when $0<\alpha_i<2$ for all $1\le i\le m$.  Then,  the following energy estimates hold for $R>\kappa^*$. 
\begin{enumerate}
\item[(i)] If $1<\alpha_*<2$, then $\mathcal E_J(u,B_R)  \le  C R^{n-1}$,
\item[(ii)] If  $\alpha_*=1$, then $\mathcal E_J(u,B_R) \le  C R^{n-1}\log R$,
\item[(iii)] If $0<\alpha_*<1$, then $\mathcal E_J(u,B_R)  \le  C R^{n-\alpha_*}$,
\end{enumerate} 
where the positive constant $C$ is independent from $R$ but may depend on $\alpha_i,\kappa_i$. 
\end{thm}
\begin{proof}
The proof is similar to the one of Theorem \ref{thmlayerK1}.   We only need to provide an upper bound for the right-hand side of (\ref{EKT}).  Note that $|u_i^t(x)- u_i^t(y)| \le C \min\{\kappa_i, |x-y|\}$. From the boundedness of $u$ we have  
\begin{eqnarray}
\mathcal E_J(u, B_R) &\le&   C  \sum_{i=1}^m  \iint_{(\mathbb{R}^n\setminus B_R )\times B_R  } \min\{\kappa_i, |x-y|\} J_i(x-y)  dy dx   
\\&\le &  C \sum_{k=1}^3 \sum_{i=1}^m  \iint_{\Pi^k_{R,\kappa_i}}  \min\{\kappa_i, |x-y|\} J_i(x-y)  dy dx =:  C \sum_{k=1}^3 \sum_{i=1}^m  L_{k,i}(R)   .  
\end{eqnarray}
Note that an upper bound for $L_{3,i}(R)$, the integral on $\Pi_{R,\kappa_i}^3$,  is similar to the one given by (\ref{G3R}) that is 
\begin{equation}\label{L3R}
L_{3,i}(R) \le C(\alpha_i,\kappa_i) R^{n-1} .
\end{equation}
Note that for subdomains $\Pi^1_{R,\kappa_i}$ and $\Pi^2_{R,\kappa_i}$, we have $\Pi^1_{R+\kappa_i,\kappa_i}\subset \Pi^2_{R,\kappa_i}$. Therefore, we only provide an upper bound for $L_{2,i}(R) $ as 
\begin{eqnarray}
L_{2,i}(R) &=&   \kappa_i \int_{B_R}    \int_{ \mathbb{R}^n\setminus B_{R+\kappa_i}(x) }   |z|^{-n-2s_i}   dz dx
\\&\le&  \kappa_i  \int_{B_{R}   } \int_{R+\kappa_i-|x|}^{\infty}   r^{-1-\alpha_i}   dr dx
\\&\le&   \frac{\kappa_i }{\alpha_i} \int_{B_{R}   }  (R +\kappa_i- |x|)^{-\alpha_i} dx   \le    \frac{\kappa_i}{\alpha_i} R^{n-1} \int_{0}^{R}     (R +\kappa_i- r)^{-\alpha_i} dr   . 
\end{eqnarray}
From this we get 
\begin{equation}\label{IntJ2R}
 L_{2,i}(R)
\le  C   \left\{ \begin{array}{lcl}
\hfill \kappa_i\log \left(\frac{R}{\kappa_i} \right)  R^{n-1}   &\text{for}&  \alpha_i=1,\\   
\hfill \frac{R_i}{\alpha_i(1-\alpha_i)} [R^{1-\alpha_i}- \kappa^{1-\alpha_i}_i] R^{n-1}   &\text{for}&  \alpha_i\neq 1 . 
\end{array}\right.
\end{equation}
This completes the proof. 

\end{proof}

\subsection{Liouville theorem for systems}
In this part, we provide a Liouville theorem for the quotient $\sigma_i:=\frac{\psi_i}{\phi_i}$ when $\psi_i:=\nabla u_i\cdot \nu$ for $\nu(x)=\nu(x',0):\RR^{n-1}\to \RR $ and $\phi=(\phi_i)_{i=1}^m$ solves the linearized system (\ref{L}).  In what follows, we first show that $\sigma=(\sigma_i)_{i=1}^m$ and $\phi=(\phi_i)_{i=1}^m$ satisfy an integral inequality.  Note that this is a counterpart of the equation (\ref{divphisigma}) for the case of local equations. 

\begin{lemma}\label{lemlinear}
Suppose that $u=(u_i)_{i=1}^m$ is a stable solution of symmetric system (\ref{main}) with a nonnegative, even and measurable kernel $J_i$. Then, for  $\eta\in C_c^1(\mathbb R^n)$  
\begin{equation}\label{sigmalin}
\sum_{i=1}^m   \int_{\RR^n} \int_{\RR^n} \sigma_i(x) [\sigma_i(x)- \sigma_i(z)] \phi_i(z) \phi_i(x)  J_i(x-z) \eta^2(x)dz dx \le 0 . 
\end{equation}
\end{lemma}
\begin{proof}
Since $u$ is a stable solution,   there exists a sequence of functions $\phi=(\phi_i)_{i=1}^m$  such that 
\begin{equation}\label{phi}
L_i(\phi_i(x))= \sum_{j=1}^m \partial_j H_i(u) \phi_j  . 
\end{equation}
Let $\nu(x)=\nu(x',0):\RR^{n-1}\to \RR $ and $\psi_i:=\nabla u_i\cdot \nu$.  Differentiating (\ref{main}) with respect to $x$,  we get 
\begin{equation}\label{psi}
L_i(\psi_i(x))= \sum_{j=1}^m \partial_j H_i(u) \psi_j. 
\end{equation}
Since $\psi_i=\sigma_i\phi_i$, from (\ref{psi}) we have 
\begin{equation}\label{sigmaphi}
L_i(\sigma_i(x)(\phi_i(x))= \sum_{j=1}^m \partial_j H_i(u) \sigma_j(x)\phi_j(x). 
\end{equation}
Multiply (\ref{phi}) with $-\sigma_i$ and add with (\ref{sigmaphi}) to get 
\begin{equation}\label{LL}
L_i(\sigma_i(x)(\phi_i(x))- \sigma_i L_i(\phi_i(x))= \sum_{j=1}^m \partial_j H_i(u) \phi_j(x)(\sigma_j(x)-\sigma_i(x)). 
\end{equation}
Applying Lemma \ref{fgprop}, we get 
\begin{equation}\label{Lphisigma}
\phi_i(x) L(\sigma_i) - \int_{\RR^n} [\sigma_i(x)-\sigma_i(z)] [\phi_i(x)-\phi_i(z)] J_i(x-z) dz = \sum_{j=1}^m \partial_j H_i(u) \phi_j(x)(\sigma_j(x)-\sigma_i(x)). 
\end{equation}  
Note that the left-hand side of the above equality can be simplified as 
\begin{equation}\label{}
 \int_{\RR^n}  [\sigma_i(x)-\sigma_i(z)] \phi_i(z) J_i(x-z) dz.
\end{equation}  
Multiplying (\ref{Lphisigma}) with $\sigma_i(x) \phi_i(x) \eta^2(x)$ and integrating,  we  get 
\begin{eqnarray}\label{lineareqsum} 
&&\sum_{i=1}^m   \int_{\RR^n} \int_{\RR^n} \sigma_i(x) [\sigma_i(x)- \sigma_i(z)] \phi_i(z) \phi_i(x)  K_i(x-z) \eta^2(x)dz dx 
\\&= & \sum_{i,j=1}^m  \int_{\RR^n}  \partial_j H_i(u) \phi_j(x) \phi_i(x) \sigma_i(x) [\sigma_j(x) -\sigma_i(x)] \eta(x)dx.
\end{eqnarray}
Note that for symmetric systems we have
\begin{eqnarray}
 \sum_{i,j} \phi_i \phi_j \partial_j H_i(u) \sigma_i (\sigma_j-\sigma_i)&=& \sum_{i< j}   \phi_i \phi_j \partial_j H_i(u)  \sigma_i (\sigma_j-\sigma_i) + \sum_{i> j}  \phi_i \phi_j \partial_j H_i(u)  \sigma_i  (\sigma_j-\sigma_i)  \\&=&\sum_{i< j}  \phi_i \phi_j \partial_j H_i(u)   \sigma_i (\sigma_j-\sigma_i)  + \sum_{i< j}  \phi_i \phi_j \partial_j H_i(u)  \sigma_j (\sigma_i-\sigma_j)  \\&=&-\sum_{i< j}  \phi_i \phi_j \partial_j H_i(u)  (\sigma_j-\sigma_i)^2 \le 0. 
  \end{eqnarray}
This completes the proof.  
 
\end{proof}

We now provide a Liouville theorem for  $\sigma=(\sigma_i)_{i=1}^m$ and $\phi=(\phi_i)_{i=1}^m$ satisfying (\ref{sigmalin}).  Note that for the case of scalar equations, $m=1$, this is given by Hamel et al.  in \cite{hrsv}.

\begin{prop}\label{thmlione}
Let $n,m\ge 1$. Suppose that $\sigma=(\sigma_i)_{i=1}^m$ and $\phi=(\phi)_{i=1}^m$ satisfy (\ref{sigmalin}) and each $\phi_i$ does not change sign.   Assume also that  $J=(J_i)_{i=1}^m$ is a sequence of measurable, nonnegative and even  kernels such that 
\begin{equation}
\sum_{i=1}^m   \iint_{ \{\cup_{k=1}^4 \Gamma^k_R \}} [\sigma_i(x) +  \sigma_i(z)]^2  \phi_i(x) \phi_i(z) |x-z|^2 J_i(x-z)   dz dx \le C R^2 , 
 \end{equation}
where $\Gamma^k_R$ are given in (\ref{gamma2}). Then, each $\sigma_i$ must be constant. 
\end{prop}

  \begin{proof} From Lemma \ref{lemlinear},   we have
    \begin{equation}\label{etasigmale0}
  \sum_{i=1}^m   \int_{\RR^n} \int_{\RR^n}  [\eta^2(x) \sigma_i(x) -  \eta^2(z) \sigma_i(z)]    [\sigma_i(x)- \sigma_i(z)] \phi_i(x) \phi_i(z)  J_i(x-z)   dz dx \le 0   , 
  \end{equation}
for a test function $\eta\in C_c^1(\mathbb R^n).$ Note that 
\begin{equation}\label{etaiden}
 2[\eta^2(x) \sigma_i(x) -  \eta^2(z) \sigma_i(z)] = [\sigma_i(x) -  \sigma_i(z)][\eta^2(x) + \eta^2(z)] 
 + [\sigma_i(x) +  \sigma_i(z)][\eta^2(x) - \eta^2(z)] .
\end{equation}
We now set to be the standard test function that is $\eta=1$ in $\overline {B_R}$ and $\eta=0$ in $\overline{\RR^n\setminus B_{2R}}$ with $||\nabla \eta||_{L^{\infty}(B_{2R}\setminus B_R)}\le C R^{-1}$. Combining (\ref{etasigmale0}) and (\ref{etaiden}), we get 
\begin{eqnarray}
0\le \Upsilon&:=& \sum_{i=1}^m   \int_{\RR^n} \int_{\RR^n}  [\sigma_i(x) -  \sigma_i(z)]^2 [\eta^2(x) + \eta^2(z)] \phi_i(x) \phi_i(z)  J_i(x-z)   dz dx 
\\&\le &    \sum_{i=1}^m   \int_{\RR^n} \int_{\RR^n}  [\sigma^2_i(z) -  \sigma^2_i(x)] [\eta^2(x) - \eta^2(z)]  \phi_i(x) \phi_i(z)  J_i(x-z)   dz dx    . 
\end{eqnarray}
Note also  that 
\begin{equation}\label{etaiden1}
 [\sigma^2_i(z) -  \sigma^2_i(x)] [\eta^2(x) - \eta^2(z)] = [\sigma_i(z) -  \sigma_i(x)] [\sigma_i(z) +  \sigma_i(x)][\eta(x) - \eta(z)] [\eta(x) + \eta(z)]   . 
 \end{equation}
Applying similar arguments as in the proof of Theorem \ref{liopo} and using the Cauchy-Schwarz inequality, we conclude  
\begin{eqnarray}
\Upsilon^2&\le & \left(\sum_{i=1}^m   \iint_{  \{\cup_{k=1}^4\Gamma^k_R \} \ }  [\sigma_i(x) -  \sigma_i(z)]^2 [\eta^2(x) + \eta^2(z)] \phi_i(x) \phi_i(z)  J_i(x-z)   dz dx  \right)
\\& & \left( \sum_{i=1}^m     \iint_{ \{\cup_{k=1}^4\Gamma^k_R\}  }  [\sigma_i(x) +  \sigma_i(z)]^2 [\eta(x) - \eta(z)]^2 \phi_i(x) \phi_i(z)  J_i(x-z)   dz dx \right)  
\\& =:& \Upsilon(R) \Theta(R) ,
\end{eqnarray}
where domain decompositions $\Gamma^k_R$  are set in (\ref{gamma2}).   From the definition of $\eta$,  for $(x,z)$ in $\{\cup_{k=1}^4\Gamma^k_R \} $    we have 
  \begin{equation}
 (\eta(x)-\eta(z))^2 \le C R^{-2} |x-z|^2  .    
   \end{equation}
Note that $\Upsilon(R)\le \Upsilon$ and from the assumptions we have 
 \begin{equation}
\Theta(R) \le R^{-2} \sum_{i=1}^m   \iint_{\cup_{k=1}^4 \Gamma^k_R} [\sigma_i(x) +  \sigma_i(z)]^2  \phi_i(x) \phi_i(z) |x-z|^2 J_i(x-z)   dz dx \le C. 
   \end{equation}
This implies that $0\le \Upsilon\le C$ and then $\Upsilon(R)\le C$. Therefore, $\Upsilon=0$.  This completes the proof. 
   
    \end{proof}
 
 We now provide a one-dimensional symmetry result for bounded $H$-monotone solutions of  system (\ref{main}) under certain assumptions on sequences $J=(J_i)_{i=1}^m$,  $\phi=(\phi_i)_{i=1}^m$ and $\sigma=(\sigma_i)_{i=1}^m$.

  \begin{thm}\label{thmapp} Suppose that $u=(u_i)_{i=1}^m$  is a bounded $H$-monotone solution of  system (\ref{main}).  Let $\phi_i:=\partial_{x_n}u_i$,  $\psi_i:=\nabla u_i\cdot \nu$ for $\nu(x)=\nu(x',0):\RR^{n-1}\to \RR $,  $\sigma_i:=\frac{\psi_i}{\phi_i}$ and the following holds
  \begin{equation}\label{sigmaR2}
\sum_{i=1}^m   \iint_{ \{\cup_{k=1}^4 \Gamma^k_R \}} [\sigma_i(x) +  \sigma_i(z)]^2  \phi_i(x) \phi_i(z) |x-z|^2 J_i(x-z)   dz dx \le C R^2 , 
 \end{equation}
when $\Gamma^k_R$ are given in (\ref{gamma2}).  Then,  each $u_i$ must be a one-dimensional function for $i=1,\cdots,m$.  
  \end{thm}
  
  \begin{proof}
  The proof is a direct consequence of Lemma \ref{lemlinear} and Proposition \ref{thmlione}.  
  \end{proof}
 
 \begin{cor} Suppose that $u=(u_i)_{i=1}^m$  is a bounded $H$-monotone solution of   (\ref{main}) in two dimensions when the jump kernel $J=(J_i)_{i=1}^m$ satisfies either (\ref{Jumpki}) or  (\ref{Jumpkei})-(\ref{Jumpkeri}) with $h_i(r)< C r^{-\theta_i}$ when $\theta_i>3$.   Assume also that the following Harnack inequality holds for each $\phi_i$ 
 \begin{equation}\label{harn}
 \sup_{B_1(x_0)} \phi_i \le C_i \inf_{B_1(x_0)} \phi_i, \ \ \text{for all} \ \ x_0\in\mathbb R^n. 
 \end{equation}
 Then,  each $u_i$ must be a one-dimensional function for $i=1,\cdots,m$.  
 \end{cor}
 
  \begin{proof}
Since $|\nabla u_i| \in L^\infty(\mathbb R^n)$, we conclude that $|\sigma_i|\le \frac{C}{\phi_i}$. This implies that 
\begin{equation}
 [\sigma_i(x) +  \sigma_i(z)]^2 \le C \left(\frac{1}{\phi^2_i(x)} + \frac{1}{\phi^2_i(z)} \right). 
\end{equation}
Therefore, 
\begin{equation}
 [\sigma_i(x) +  \sigma_i(z)]^2  \phi_i(x) \phi_i(z)  \le C  \left( \frac{\phi_i(x)}{\phi_i(z)} + \frac{\phi_i(z)}{\phi_i(x)}  \right). 
\end{equation}
Note that the Harnack inequality (\ref{harn}) implies that 
\begin{equation}
 [\sigma_i(x) +  \sigma_i(z)]^2  \phi_i(x) \phi_i(z)  \le C. 
\end{equation}
Applying similar arguments as in the proof of Theorem \ref{liopo}, one can conclude that (\ref{sigmaR2}) holds in two dimensions. This completes the proof.   
  \end{proof}

We end this section with pointing out that we expect that the one-dimensional symmetry result holds in three dimensions as well. To prove this, just like in the proof of the De Giorgi's conjecture given by Ambrosio and Cabr\'{e} \cite{ac}, one might need to combine the energy estimate provided in Theorem \ref{thmlayerK1} and Theorem \ref{thmapp}.

\vspace*{.4 cm}

\noindent {\it Acknowledgment}.   The first author would like to thank Yannick Sire for his comments and online discussions. And he is thankful to ETH Zurich for the hospitality where parts of this work were completed.

\end{document}